\documentclass{amsart}

\usepackage{float}
\usepackage{graphics}
\usepackage{amssymb,amsmath,amsthm,amscd}
\usepackage{mathrsfs}
\usepackage{enumerate}
\usepackage{color}
\usepackage[vcentermath,enableskew,noautoscale]{youngtab}
\usepackage[all]{xy}
\usepackage{verbatim}
\usepackage{multicol}

\newtheorem{theorem}{Theorem}[section]
\newtheorem{lemma}[theorem]{Lemma}
\newtheorem{proposition}[theorem]{Proposition}
\newtheorem{corollary}[theorem]{Corollary}

\theoremstyle{definition}
\newtheorem{definition}[theorem]{Definition}
\newtheorem{example}[theorem]{Example}

\theoremstyle{remark}
\newtheorem{remark}[theorem]{Remark}

\numberwithin{equation}{section}

\newcommand{\Prop}{\begin{proposition}}
\newcommand{\enprop}{\end{proposition}}
\newcommand{\Lemma}{\begin{lemma}}
\newcommand{\enlemma}{\end{lemma}}
\newcommand{\Th}{\begin{theorem}}
\newcommand{\enth}{\end{theorem}}
\newcommand{\Cor}{\begin{corollary}}
\newcommand{\encor}{\end{corollary}}
\newcommand{\Def}{\begin{definition}}
\newcommand{\edf}{\end{definition}}
\def\beq{\begin{equation*}}
\def\eeq{\end{equation*}}
\def\bea{\begin{aligned}}
\def\eea{\end{aligned}}
\def\bee{\begin{enumerate}}
\def\eee{\end{enumerate}}

\let\goth\mathfrak

\def\gh{\goth h}

\def\gb{\goth b}

\def\gn{\goth n}

\def\gq{\goth q}

\def\nn{\nonumber}

\def\cplus{\hbox{$\supset${\raise1.05pt\hbox{\kern -0.55em
${\scriptscriptstyle +}$}}\ }}

\def\gdiw{\Lambda^{+}_{\bar 0}}
\def\qdiw{\Lambda^{+}}
\def\Ao{\mathbf{A}_1}
\def\Jo{\mathbf{J}_1}

 \DeclareMathOperator{\ch}{ch}

\DeclareMathOperator{\Cliff}{Cliff}

\def\Z{\mathbb Z}
\def\C{\mathbb C}

\def\al{\alpha}

\def\ep{\epsilon}

\def\la{\lambda}

\def\tensor{\otimes}
\def\op{\oplus}

\def\ra{\longrightarrow}
\def\map{\longmapsto}

\def\dd{\delta_{ij}}

\def\ol{\overline}

\def\F{\mathbf F}
\def\A{\mathbf A}

\newcommand{\V}{{\mathbf V}}

\newcommand{\seteq}{\mathbin{:=}}

\def\nc{\newcommand}

\nc{\te}{\tilde{e}} \nc{\tei}{\tilde{e}_i} \nc{\tf}{\tilde{f}}
\nc{\tfi}{\tilde{f}_i} \nc{\tU}{\widetilde U_q(\g)}
\nc{\tE}{\tilde{E}} \nc{\tF}{\tilde{F}}

\nc{\tkone}{\tilde{k}_{\ol{1}}} \nc{\teone}{\tilde{e}_{\ol{1}}}
\nc{\tfone}{\tilde{f}_{\ol{1}}}

\nc{\teibar}{\tilde{e}_{\ol{i}}} \nc{\tfibar}{\tilde{f}_{\ol{i}}}
\nc{\tki}{\tilde{k}_{\ol {i}}}
\newcommand{\isoto}[1][]{\mathop{\xrightarrow[#1]%
{{\raisebox{-.6ex}[0ex][-.6ex]{$\mspace{2mu}\sim\mspace{2mu}$}}}}}

\nc{\Uq}{U_q(\goth{q}(n))} \nc{\B}{\mathbf B}
\nc{\BZ}{{\mathbb{Z}}}

\nc{\qs}{{q}} \nc{\lan}{\langle} \nc{\ran}{\rangle}
\nc{\re}{{\mathrm{re}}}

\nc{\Uf}[1][\g]{U^-_q(#1)} \nc{\Ue}{U^+_q(\g)}
\nc{\eps}{\varepsilon} \nc{\vphi}{\varphi} \nc{\sphi}{\varphi^*}
\nc{\seps}{\varepsilon^*}

\def\max{{\mathop{\mathrm{max}}}}
\nc{\vp}{\varpi} \nc{\cls}{{\operatorname{cl}}} \nc{\Wt}{{\rm Wt}}
\nc{\wt}{{\rm wt}} \nc{\Us}{U'_q(\g)}

\nc{\ro}{{\rm(}} \nc{\rf}{{\rm)}} \nc{\qn}{\goth{q}(n)}

\nc{\zb}{\ol{0}} \nc{\ob}{\ol{1}} \nc{\ib}{\ol{i}}
\nc\ghzb{\gh_{\zb}} \nc\ghzbs{\ghzb^*}

\def\La{\Lambda}
\newcommand{\soplus}{\mathop{\mbox{\normalsize$\bigoplus$}}\limits}
\nc\Oint{\mathcal{O}_{int}^{\ge 0}}
\newcommand{\bna}{\begin{enumerate}[{\rm(1)}]}
\nc\ena{\end{enumerate}} \nc{\bni}{\begin{enumerate}[{\rm(i)}]}
\nc\uqqn{U_q(\qn)}
\newcommand{\eq}{\begin{eqnarray}}
\newcommand{\eneq}{\end{eqnarray}}
\def\read{{\rm read}}

\begin{document}

\title{Quantum queer superalgebras}

\author{Ji Hye Jung}
\address{Department of Mathematical Sciences and Research Institute of Mathematics \\
         Seoul National University \\ Seoul 151-747, Korea}
\curraddr{Department of Mathematical Sciences, Korea Advanced Institute of Science and Technology, Daejeon 305-701, Korea}
\email{jhjung@math.snu.ac.kr, jihyejung@kaist.ac.kr}
\thanks{The first author was partially supported by BK21 Mathematical Sciences Division and NRF Grant \# 2010-0019516.}

\author{Seok-Jin Kang}
\address{Department of Mathematical Sciences and Research Institute of Mathematics \\
         Seoul National University \\ Seoul 151-747, Korea}
\email{sjkang@math.snu.ac.kr}
\thanks{The second author was partially supported by KRF Grant \# 2007-341-C00001 and NRF Grant \# 2010-0010753.}





\begin{abstract}
We give a brief survey of recent developments in the highest
weight representation theory and the crystal basis theory of the
quantum queer superalgebra $\Uq$.
\end{abstract}

\maketitle

\section*{Introduction}
In this expository article, we give an elementary account of recent
developments in the highest weight representation theory and the
crystal basis theory of {\it quantum queer superalgebra} $\Uq$. The
{\it queer Lie superalgebra} $\qn$ has attracted a great deal of
research activities due to its resemblance to the general linear Lie
algebra $\mathfrak{gl}(n)$ on the one hand and its unique features
in its structure and representation theory on the other hand. The
Lie superalgebra $\qn$ is similar to $\mathfrak{gl}(n)$ in that the
tensor powers of natural representations are all completely
reducible. Moreover, there is a queer analogue of the celebrated
{\it Schur-Weyl duality}, often referred to as the {\it
Schur-Weyl-Sergeev duality}, that was discovered in \cite{O, Ser}.
However, this is about the end of their resemblance and there is a
vast list of differences and discrepancies between these two
algebraic structures. One of the major difficulties lies in that the
Cartan subalgebra of $\qn$ is not abelian and has a nontrivial odd
part. For this reason, it is a very complicated and challenging task
to investigate the structure and representation theory of queer Lie
superalgebra $\qn$ (see, for example, \cite{B, G, K, P, PS, Ser1,
Ser}). Thus a {\em queer} version of the crystal basis theory would
be very helpful in understanding the combinatorial representation
theory of $\gq(n)$.

A quantum deformation $\Uq$ of the universal enveloping algebra
$U(\qn)$ was constructed by Olshanski \cite{O} using a modification
of the Reshetikhin-Takhtajan-Faddeev method \cite{RTF}. In
\cite{GJKK}, Grantcharov, Jung, Kang and Kim gave a presentation of
$\Uq$ in terms of Chevelley generators and Serre relations and
developed the highest weight representation theory of $\Uq$ with a
door open to the crystal basis theory. The authors of \cite{GJKK}
defined the category $\mathcal{O}_{int}^{\ge 0}$, and proved the
{\it classical limit theorem} and the {\it complete reducibility
theorem}. Since the queer Lie superalgebra $\qn$ has a nontrivial
odd Cartan part which is closely related with the Clifford algebra,
the highest weight space of every finite dimensional $\qn$-module
admits a structure of a Clifford module. In \cite{GJKK}, a complete
classification of irreducible quantum Clifford modules was also
given.

In \cite{GJKKK1, GJKKK2}, Grantcharov, Jung, Kang, Kashiwara and Kim
developed the crystal basis theory for $\Uq$-modules in the category
${\mathcal O}_{int}^{\geq 0}$. The authors of \cite{GJKKK1, GJKKK2}
first enlarge the base field to $\C((q))$, the field of formal
Laurent power series and obtain an equivalence of the categories of
Clifford modules and quantum Clifford modules, which yields a
standard version of classical limit theorem. As the next step, they
introduced the {\it odd Kashiwara operators} $\tilde{e}_{\ol{1}}$,
$\tilde{f}_{\ol{1}}$, and $\tilde{k}_{\ol{1}}$, where
$\tilde{k}_{\ol{1}}$ corresponds to an odd element in the Cartan
subsuperalgebra of $\qn$. A {\it crystal basis} for a
$U_q(\qn)$-module $M$ in the category ${\mathcal O}_{int}^{\geq 0}$
is defined to be a triple $(L,B,(l_b)_{b \in B})$, where the crystal
lattice $L$ is a free $\C[[q]]$-submodule of $M$, $B$ is a finite
$\mathfrak{gl} (n)$-crystal,  $(l_b)_{b \in B}$ is a family of
non-zero subspaces of $L/qL$ such that $L / qL = \bigoplus_{b \in B}
l_b$, with a set of compatibility conditions for the action of the
Kashiwara operators. The {\it queer tensor product rule} for odd
Kashiwara operators is very different from the usual ones and is
quite interesting. The main result of \cite{GJKKK1, GJKKK2} is the
existence and the uniqueness theorem for crystal bases. One of the
key ingredients of the proof is the characterization of highest
weight vectors in ${\bf B} \otimes B(\lambda)$ in terms of even
Kashiwara operators and the highest weight vector of $B(\lambda)$.
All these statements are verified simultaneously by a series of
interlocking inductive arguments.

 In \cite{GJKKK3}, Grantcharov, Jung, Kang, Kashiwara and
Kim gave an explicit combinatorial realization of the crystal
$B(\lambda)$ for an irreducible highest weight module $V^q(\lambda)$
in terms of {\it semistandard decomposition tableaux}. A class of
combinatorial objects that describe the tensor representations of
${\mathfrak q}(n)$ has been known for more than thirty years - the
{\it shifted semistandard Young tableaux}. These objects have been
extensively studied  by Sagan, Stembridge, Worley, and others,
leading to important and deep results (in particular, the {\em
shifted Littlewood-Richardson rule}) \cite{Sag, St, Wor}. However,
the set of shifted semistandard Young tableaux of a fixed shape does
not have a natural crystal structure. For this reason, in
\cite{GJKKK3}, it was necessary to use seimistandard decomposition
tableaux instead of shifted semistandard Young tableaux. Moreover,
the authors of \cite{GJKKK3} presented a queer crystal version of
{\it insertion scheme} and proved another version of the shifted
Littlewood-Richardson rule for decomposing the tensor product
$B(\la) \tensor B(\mu)$ for all strict partitions $\la, \mu$. The
insertion scheme in \cite{GJKKK3} is analogous to the one introduced
in \cite{Serra} and can be considered as a variation of those used
for shifted tableaux by Fomin, Haiman, Sagan, and Worley \cite{Fom,
Hai, Sag, Wor}. Consequently, the results of \cite{GJKKK3} establish
a combinatorial description of the {\em shifted
Littlewood-Richardson coefficients}. It is expected that the queer
crystal basis theory will shed a new light on a wide variety of
interesting combinatorics.

In this paper, we do not give any proof. Instead, we only give the
main idea of proofs and some relevant remarks.

\vskip3mm
\section{Queer Lie superalgebra $\gq(n)$} \label{Queer Lie superalgebra}

We begin with the definition of queer Lie superalgebra $\gq(n)$.

\begin{definition}
The {\em queer Lie superalgebra} $\gq(n)$ is the Lie superalgebra
over $\C$ defined in matrix form by
$$
\gq(n):= \left\{ \left( \begin{array}{cc} A& B \\ B & A
\end{array} \right)\; \Big| \; A,B \in \goth{gl}(n,\C) \right\}=\qn_{\zb} \soplus \qn_{\ob},
$$
where
$$
\qn_{\bar{0}}:= \left\{ \left( \begin{array}{cc} A& 0 \\ 0 & A
\end{array} \right) \right\}, ~~ \qn_{\bar{1}}:= \left\{ \left( \begin{array}{cc} 0& B \\ B & 0
\end{array} \right) \right\}.
$$
\end{definition}
The superbracket is defined to be
$$[x,y]=xy-(-1)^{\alpha \beta }yx \quad \text{for} \ \alpha,\beta \in \Z_2 \ \text{and} \ x \in \qn_{\alpha}, y \in
\qn_{\beta}.$$ The (standard) {\em Cartan subalgebra} $\gh
=\gh_{\bar 0} \oplus \gh_{\bar 1}$ is given by
$$\gh_{\zb} =\C k_1 \op \cdots \op \C k_n  \ \text{and} \ \gh_{\ob} =\C k_{\ob} \op \cdots \op
\C k_{\bar n},$$ where
$$ k_i:=  \left( \begin{array}{cc} E_{i,i}& 0 \\  0 & E_{i,i}
\end{array} \right), \quad k_{\ib}=\left(\begin{array}{cc}
    0 & E_{i,i} \\
    E_{i,i} & 0 \\
\end{array}\right),  $$
and $E_{i,j }$ is the $n \times n$ matrix having $1$ at the
$(i,j)$-entry and $0$ elsewhere. Note that the Cartan subalgebra
$\gh$ has a nontrivial odd part $\gh_{\ob}$, and hence $\gh$ is
not abelian.

For $i=1, \ldots, n-1,$ set
$$e_i= \left( \begin{array}{cc} E_{i,i+1}& 0 \\  0 & E_{i,i+1}
\end{array} \right), \quad e_{\ib}=\left( \begin{array}{cc} 0 & E_{i,i+1} \\  E_{i,i+1} &
0
\end{array} \right), $$
and
$$f_i= \left( \begin{array}{cc} E_{i+1,i}& 0 \\  0 & E_{i+1,i}
\end{array} \right), \quad f_{\ib}=\left( \begin{array}{cc} 0 & E_{i+1,i} \\  E_{i+1,i} &
0
\end{array} \right). $$

 Let $\{ \epsilon_1, \ldots, \epsilon_n \}$ be the basis of
$\gh_{\zb}^*$ such that $\epsilon_i(k_j)=\delta_{ij}$ and
$\alpha_i=\epsilon_i-\epsilon_{i+1}$ be the {\em simple roots} for
$i=1, \ldots, n-1$.

  \begin{proposition}\cite[\S 3]{LS}
 The queer Lie superalgebra $\gq(n)$ is generated by the elements $e_i, e_{\ib}, f_i,
  f_{\ib}$ $(i=1, \ldots, n-1)$, $\gh_{\zb}$ and $k_{\ol j}$ $(j =1, \ldots, n)$ with the
  following defining relations:
\begin{align} \label{defining relations of q(n)}
\nonumber & [h,h']=0 \  \ \text{for} \ \ h, h' \in
\gh_{\zb},\\
\nn & [h,e_i]=\al_i(h)e_i, \ [h, f_i]=-\al_i(h)f_i \ \ \text{for} \  \ h \in \gh_{\zb}, \\
\nonumber & [h,k_{\bar{j}}]=0 \ \ \text{for} \ \ h \in \gh_{\zb}, \\
\nn & [e_i, f_j]=\delta_{ij} (k_i-k_{i+1}), \\
\nn & [e_i,e_j]=[f_i, f_j]=0 \ \ \text{if} \ \ |i-j|>1, \\
\nn & [k_{\bar i}, k_{\bar j}]=\dd 2k_{i}, \\
  &[e_i,f_{\bar{j}}]=\dd(k_{\bar{i}}-k_{\overline{i+1}}),
 \ [e_{\bar i},f_j]=\dd(k_{\bar i}-k_{\overline{i+1}}), \\
\nonumber &  [k_{\overline j},e_i]=\al_i(k_{j})e_{\overline i}, \
[k_{\overline j},f_i]=- \al_i(k_j) f_{\overline i}, \\
\nn & [e_i,e_{\bar j}]  = [e_{\ol i}, e_{\ol j}] = [f_i,f_{\bar j}]=[f_{\ol i}, f_{\ol j}]=0 \ \ \text{if} \ \ |i-j| \neq 1,  \\
\nn & [e_{i}, e_{i+1}] = [e_{\bar{i}}, e_{\overline{i+1}}], [e_{i}, e_{\overline{i+1}}] = [e_{\bar{i}}, e_{i+1}], \\
\nn & [f_{i+1}, f_{i}] = [f_{\overline{i+1}}, f_{\bar{i}}],  [f_{i+1}, f_{\bar{i}}]=   [f_{\overline{i+1}}, f_{i}], \\
 \nn & [e_i, [e_i, e_j]]=[f_i, [f_i,f_j]]=0 \ \ \text{if} \ \
|i-j|=1, \\
\nn & [e_{\bar i}, [e_i, e_j]]=[f_{\bar i}, [f_i, f_j]]=0 \ \
\text{if} \ \
 |i-j|=1.
\end{align}
 \end{proposition}

The elements $e_i, f_i$ ($i=1, \ldots, n-1$) and $h \in \gh_{\zb}$
are regarded as {\em even} generators, and the elements $e_{\ib},
f_{\ib}$ ($i=1, \ldots, n-1$) and $k_{\ol j}$ ($j=1, \ldots, n$)
are regarded as {\em odd} generators. One can see that the
relations involving $e_i, f_i, h$ for $h \in \gh_{\zb}$ are the
same as the relations for the general linear Lie algebra
$\goth{gl}(n)$.

\begin{remark}
We have the relations
$$[k_{\bar i}, e_i] =e_{\bar i}, \ [k_{\bar i},
f_i] =-f_{\bar i}, \ \ \text{and} \ \ [e_i,f_{\bar i}]=k_{\bar
i}-k_{\overline{i+1}}=[e_{\bar i}, f_i].$$
From these relations, it is easy to see that the queer Lie
superalgebra $\gq(n)$ is generated by $e_i, f_i$ $(i=1,
\ldots,n-1)$, $\gh_{\zb}$ and $k_{\bar 1}$ only.
\end{remark}

The universal enveloping algebra $U(\qn)$ of $\qn$ is constructed
from the
tensor algebra $T (\qn)$  
 by factoring out by the ideal generated by the
elements $[u,v]-u \tensor v+(-1)^{\alpha \beta} v \tensor u$,
where $\alpha, \beta \in \Z_2, u \in \qn_{\alpha}, v \in
\qn_{\beta}$. Let $U^+$ (respectively, $U^-$) be the subalgebra
$U(\qn)$ generated by $e_i, e_{\ib}$ (respectively, $f_i,
f_{\ib}$) for $i=1, \ldots, n-1$, and let $U^0$ be the subalgebra
generated by $k_j, k_{\ol j}$ for $j =1, \ldots, n$. By the
Poincar\'e-Birkhoff-Witt theorem in \cite{MM}, we obtain the
triangular decomposition of $U(\qn)$:
$$U(\qn) \cong U^- \tensor U^0 \tensor U^+. $$

\vskip 3mm

\section{Highest weight modules over $\qn$}

 Recall that $\gh_{\zb}=\C k_1 \op \cdots \op \C k_n$, and $\{
\epsilon_1, \ldots, \epsilon_n \}$ is the basis of $\gh_{\zb}^*$
dual to the basis $\{ k_1, \ldots, k_n \}$ of $\gh_{\zb}$.
Let $P :=\Z \epsilon_1 \op \cdots \op \Z \epsilon_n$ be the {\em
weight lattice} and $P^{\vee} :=\Z k_1 \op \cdots \op \Z k_n$ be
the {\em dual weight lattice}.

 \begin{definition}
 Let $\gdiw$ and $\qdiw$ be the set of $\goth{gl}(n)$-{\em
dominant integral weights} and the set of $\gq(n)$-{\em dominant
integral weights} given as follows:
\begin{align}
\nonumber &\gdiw := \{ \la_1\epsilon_1+\cdots+\la_n \epsilon_n\in
\ghzbs ~|~ \la_i-\la_{i+1} \in \Z_{\geq 0} \  \text{for all} \ i =1, \ldots, n-1 \},   \\
 \nonumber &\qdiw := \{ \la_1\epsilon_1+\cdots+\la_n
\epsilon_n\in \gdiw ~|~ \la_i=\la_{i+1} \Rightarrow
\la_i=\la_{i+1}=0 \ \text{for all} \ i=1, \cdots, n-1 \}.
\end{align}
\end{definition}

From now on, for a superalgebra $A$, an $A$-module will be
understood as an $A$-supermodule. A $\qn$-module $V$ is called a
{\em weight module} if it admits a weight space decomposition
$$    V = \bigoplus_{\mu \in \gh_{\bar 0}^*} V_{\mu},
       ~~\text{where}~~ V_{\mu} =\{ v \in V ~|~ h v = \mu(h)v ~~\text{for all}~~ h \in \gh_{\bar 0} \}. $$
For a weight $\qn$-module $V$, we denote by $\wt(V)$ the set of
$\mu \in \gh_{\zb}^*$ such that $V_{\mu} \neq 0$. If $\dim_{\C}
V_{\mu} < \infty$ for all $\mu \in \ghzbs$, the {\em character} of
$V$ is defined to be
  $$\ch V=\sum_{\mu \in \ghzbs} (\dim_{\C}V_{\mu})e^{\mu},  $$
where $e^{\mu}$ are formal basis elements of the group algebra
$\C[\ghzbs]$ with the multiplication $e^{\la}e^{\mu}=e^{\la+\mu}$
for all $\la, \mu \in \ghzbs$.

\begin{definition} A weight module $V$ is called a {\em highest weight module
with highest weight $\la \in {\mathfrak h}_{\ol 0}^*$} if
$V_{\la}$ is finite-dimensional and satisfies the following
conditions:
\begin{enumerate}[{\rm (1)}]
\item $V$ is generated by $V_{\la}$,
\item $e_i v = e_{\ol i} v =0$ for all $v \in V_{\la}$,
$i=1, \ldots, n-1$.
\end{enumerate}   \end{definition}

Note that the highest weight space of a highest weight module is
not one-dimensional.

Let $\gb_+$ be the (standard) {\it Borel subalgebra} of $\qn$
generated by $e_i, e_{\ib}$ ($i=1, \ldots, n-1$) and $k_j, k_{\ol
j }$ for $j=1, \ldots, n$. For $\la \in \ghzbs$, let ${\rm
Cliff}(\la)$ be the associative superalgebra over $\C$ generated
by the odd generators $\{t_{\ol i} \ | \ i=1,2,\ldots, n \}$ with
the defining relations
$$
\begin{array}{cc}
t_{\ol i} t_{\ol j} + t_{\ol j} t_{\ol i} = 2 \delta_{ij} \la_i, &
i,j = 1,2,\ldots,n.
\end{array}
$$

The following propositions are well-known.

\begin{proposition} \cite[Table 2]{ABS} \label{prop:ABS}
The superalgebra ${\rm Cliff}(\la)$ 
has up to isomorphism \begin{enumerate}[{\rm (1)}]
\item two irreducible modules $E(\la)$ and $\Pi(E(\la))$ of dimension $2^{k-1} | 2^{k-1}$ if $m=2k$,
\item one irreducible module $E(\la) \cong \Pi(E(\la))$ of dimension $2^k | 2^k$ if $m=2k+1$,
\end{enumerate}
where $m$ is the number of non-zero parts of $\la \in \ghzbs$ and
$\Pi$ is the parity change functor.
\end{proposition}

\begin{proposition} \label{prop:penkov} \cite[Proposition 1]{P}
Let ${\bf v}$ be a finite-dimensional irreducible $\Z_2$-graded
$\gb_+$-module.
\begin{enumerate}[{\rm (1)}]
\item The maximal nilpotent subalgebra $\gn_+$ of $\gb_+$ acts on ${\bf v}$ trivially.
\item There exists a unique weight $\lambda \in \gh_{\bar 0}^*$ such that ${\bf v}$
is a $\Z_{2}$-graded ${\rm Cliff}(\lambda)$-module.
\item  For all $h \in \gh_{\bar 0}, v \in {\bf v}$, we have $h v =\la(h)v$.
\end{enumerate}
\end{proposition}

By Proposition \ref{prop:ABS} and Proposition \ref{prop:penkov},
we get a complete classification of finite-dimensional irreducible
$\gb_+$-modules.

\begin{definition}
Let $\bf{v}(\lambda)$ be a finite-dimensional irreducible
$\gb_+$-module determined by $\lambda$. The {\em Weyl module}
$W(\lambda)$ corresponding to $\lambda$ is defined to be
$$W(\lambda) := U(\gq(n)) \otimes_{U(\gb_+)} \bf v (\lambda).$$
\end{definition}
\noindent Note that $W(\lambda)$ is defined up to $\Pi$.

\begin{theorem} \label{thm:P} \cite[Theorem 2, 4]{P}
\begin{enumerate}[{\rm (1)}]
\item For any weight $\lambda$, $W(\lambda)$ has a unique maximal
submodule $N(\lambda)$.
\item For each finite-dimensional irreducible $\gq(n)$-module $V$,
there exists a unique weight $\lambda \in \gdiw$ such that $V$ is
a homomorphic image of $W(\lambda)$.
\item The irreducible quotient $V(\lambda):=W(\lambda)/N(\lambda)$ is finite-dimensional if and only if $\lambda \in
\qdiw$.
\end{enumerate}
\end{theorem}

Set $P^{\geq 0} = \{ \la=\la_1 \ep_1 + \cdots + \la_n \ep_n \in P
\ | \ \la_j \geq 0 \quad \text{for all} \ j=1,2,\ldots,n \}$.

\begin{definition}
 The {\em category
$\mathcal{O}^{\geq 0}$} consists of finite-dimensional
$U(\gq(n))$-modules $M$ with a weight space decomposition
satisfying the following conditions:

{\rm (1)} $\wt(M) \subset P^{\ge 0}$,

{\rm (2)} if $\lan k_i,\mu\ran=0$ for $\mu\in P^{\ge0}$ and $i \in
\{1,\ldots, n\}$, then $k_{\ol i}$ acts trivially on $M_{\mu}$.
\end{definition}

The category $\mathcal{O}^{\geq 0}$ is closed under finite direct
sum, tensor product and taking submodules and quotient modules.

\begin{proposition} \label{prop:GJKK} \cite[Proposition 1.6, 1.8, 1.9]{GJKK}
\begin{enumerate}[{\rm (1)}]
\item For each $\lambda \in \qdiw \cap P^{\geq 0}$, the irreducible quotient
$V(\lambda)=W(\lambda)/N(\lambda)$ lies in the category ${\mathcal
O}^{\geq 0}.$

\item Every irreducible $U(\gq(n))$-module in the category ${\mathcal
O}^{\geq 0}$ has the form $V(\lambda)$ for some $\lambda \in \qdiw
\cap P^{\geq 0}$.

\item If $V$ is a finite-dimensional highest weight module
with highest weight $\lambda \in \Lambda^+ \cap P^{\geq 0}$ and
$V_{\la}$ is an irreducible $\gb_+$-submodule of V, then $V \simeq
V(\lambda)$ (up to $\Pi$).

\item If $V$ is a highest weight module with highest weight
$\lambda \in \Lambda^{+}$ and $f_i^{\lambda(h_i)+1} v = 0$ for all
$v \in V_{\lambda}$, $i =1,2,\ldots, n-1$, then $\dim V < \infty$.
\end{enumerate}
\end{proposition}

Note that every element $\la$ of $\qdiw \cap P^{\geq 0}$ is the
form
$$\la_1 > \la_2 > \cdots > \la_r > \la_{r+1}= \cdots =\la_n=0$$ for
some $r$. Hence we can identify an element $\lambda$ of $\qdiw
\cap P^{\geq 0}$ with a strict partition. We denote by $\ell(\la)=r$
and $|\la|=\la_1 + \cdots +\la_r$.

\vskip 3mm

\section{Quantum queer superalgebra $U_q(\gq(n))$}

Let $\F=\C((q))$ be the field of formal Laurent series in an
indeterminate $q$ and let $\A=\C[[q]]$ be the subring of $\F$
consisting of formal power series in $q$. For $k \in \Z_{\ge 0}$,
we define
$$[k]= \frac{q^k - q^{-k}}{q - q^{-1}}, \quad [0]!=1, \quad [k]! =
[k] [k-1] \cdots [2][1].$$ In \cite[\S 4]{O}, Olshanski constructed
a quantum deformation $U_q(\gq(n))$ of $U(\qn)$ using a modification
of the Reshetikhin-Takhtajan-Faddeev method. In \cite[Theorem
2.1]{GJKK}, based on Olshanski's construction, we obtain the
following presentation of $\Uq$, which is taken to be the
definition.

\begin{definition}
   The {\em quantum queer superalgebra} $U_{q}(\gq(n))$ is an $\F$-superalgebra
generated by the elements $e_i,e_{\ol{i}}, f_i, f_{\ol{i}}$, $(i =
1,...,n-1)$, $k_{\bar{j}}$, $(j = 1,...,n)$ and $q^h$ $(h \in
P^{\vee})$ with the following defining relations:
\begin{align}  \label{eq:defining relations for Uqqn}
\allowdisplaybreaks \nonumber & q^{0}=1, \ \ q^{h_1} q^{h_2} =
q^{h_1 + h_2}  \quad \text{for} \ h_1,
h_2 \in P^{\vee}, \\
\nonumber & q^h e_i q^{-h} = q^{\alpha_i(h)} e_i \quad \text{for} \ h \in P^{\vee}, \displaybreak[1]\\
\nonumber & q^h f_i q^{-h} = q^{-\alpha_i(h)} f_i  \quad \text{for} \ h\in P^{\vee}, \displaybreak[1]\\
\nonumber & q^h k_{\ol j} = k_{\ol j} q^h  \quad \text{for} \ h\in P^{\vee}, \displaybreak[1]\\
\nonumber & e_i f_j - f_j e_i = \delta_{ij} \dfrac{q^{k_i -
k_{i+1}} - q^{-k_i
+ k_{i+1}}}{q-q^{-1}}, \displaybreak[1]\\
\nonumber & e_i e_j - e_j e_i = f_i f_j - f_j f_i = 0 \quad \text{if} \ |i-j|>1, \displaybreak[1]\\
\nonumber & e_i^2 e_j -(q+q^{-1}) e_i e_j e_i  + e_j e_i^2= 0  \quad \text{if} \ |i-j|=1,\displaybreak[1]\\
\nonumber & f_i^2 f_j - (q+q^{-1}) f_i f_j f_i + f_j f_i^2 = 0  \quad \text{if} \ |i-j|=1,\displaybreak[1]\\
\nonumber & k_{\ol i}^2 = \dfrac{q^{2k_i} - q^{-2k_i}}{q^2 - q^{-2}}, \displaybreak[1]\\
 & k_{\ol i} k_{\ol j} + k_{\ol j} k_{\ol i} =0 \quad
\text{if} \ i \neq j, \displaybreak[1]\\
\nonumber & k_{\ol i} e_i - q e_i k_{\ol i} = e_{\ol i} q^{-k_i}, \ q k_{\ib}e_{i-1}- e_{i-1}k_{\ib}=-q^{-k_i} e_{\ol{i-1}}, \displaybreak[1]\\
\nn & k_{\ib}e_j-e_jk_{\bar i}=0 \quad \text{if} \ j \neq i,i-1, \displaybreak[1]\\
\nn & k_{\ol i} f_i - q f_i k_{\ol i} = -f_{\ol i} q^{k_i}, \ q k_{\ib} f_{i-1}-f_{i-1}k_{\ib}=q^{k_i}f_{\ol{i-1}}, \displaybreak[1]\\
\nn & k_{\ib}f_j-f_jk_{\ib}=0 \quad \text{if} \ j \neq i, i-1, \displaybreak[1]\\
 \nonumber & e_i f_{\ol j} - f_{\ol j} e_i = \delta_{ij}
(k_{\ol i}
q^{-k_{i+1}} - k_{\ol{i+1}} q^{-k_i}), \displaybreak[1] \\
\nonumber & e_{\ol i} f_j - f_j e_{\ol i} = \delta_{ij} (k_{\ol i}
q^{k_{i+1}} - k_{\ol{i+1}} q^{k_i}), \displaybreak[1]\\
\nonumber &e_i e_{\ol i} - e_{\ol i} e_i = f_i f_{\ol i} - f_{\ol i} f_i = 0, \displaybreak[1]\\
\nonumber &e_i e_{i+1} - q e_{i+1}e_i =
e_{\overline{i}}e_{\overline{i+1}}+ q
e_{\overline{i+1}}e_{\overline{i}}, \displaybreak[1]\\
\nonumber& q f_{i+1}f_i - f_i f_{i+1} =
 f_{\overline{i}}f_{\overline{i+1}}+ q f_{\overline{i+1}}f_{\overline{i}}, \displaybreak[1]\\
\nonumber & e_i^2 e_{\overline{j}} - (q+q^{-1})e_i
e_{\overline{j}}
e_i + e_{\overline{j}} e_i^2= 0 \quad \text{if} \ |i-j|=1, \displaybreak[1]\\
\nonumber & f_i^2 f_{\overline{j}} - (q+q^{-1})f_i
f_{\overline{j}} f_i + f_{\overline{j}} f_i^2=0 \quad \text{if} \
|i-j|=1.
\end{align}
\end{definition}

The generators $e_i$, $f_i$ $(i=1, \ldots, n-1)$, $q^{h}$ ($h\in
P^\vee$) are regarded as {\em even} and $e_{\ol i}$, $f_{\ol i}$
$(i=1, \ldots, n-1)$, $k_{\ol j}$ $(j=1, \ldots, n)$ are {\em
odd}. From the defining relations, we can see that the even
generators together with $k_{\ol 1}$ generate the whole algebra
$\Uq$.

In \cite[\S 4]{O}, Olshanski showed that the quantum queer
superalgebra $U_q(\qn)$ is a Hopf superalgebra. The comultiplication
$\Delta$ is given as follows:
\begin{equation} \label{eq:comultiplication}
\begin{aligned}
& \Delta(q^{h})  = q^{h} \otimes q^{h}\quad\text{for $h\in P^\vee$,} \\
& \Delta(e_i)  = e_i \otimes q^{-k_i + k_{i+1}} + 1 \otimes e_i, \\
& \Delta(f_i)  = f_i \otimes 1 + q^{k_i - k_{i+1}} \otimes f_i, \\
& \Delta(k_{\ol 1}) =k_{\ol 1}\otimes q^{k_1}+ q^{-k_1} \otimes
k_{\ol 1}.
\end{aligned} \end{equation}

 Let $U_q^{+}$ (respectively, \ $U_q^{-}$) be the subalgebra of
$U_q(\mathfrak{q}(n))$ generated by $e_i$, $e_{\ol i}$
(respectively, \ $f_i$, $f_{\ol i}$) for $i=1,\ldots, n-1$, and
let $U_q^{0}$ be the subalgebra generated by $q^{h}$ and $k_{\ol
j}$ for $h\in P^\vee$, $j=1, \ldots, n$.  Then we obtain the
following {\em triangular decomposition of $U_q(\qn)$.}

\begin{proposition}\cite[Theorem 2.3]{GJKK}
There is a $\mathbb{C}((q))$-linear isomorphism
$$ U_q(\gq(n)) \simeq U_q^- \otimes U_q^0 \otimes U_q^+.$$
\end{proposition}

\begin{proof} The proof is based on the comultiplication
\eqref{eq:comultiplication}, and follows the outline given in
\cite[Theorem 3.1.5]{HK2002}.
\end{proof}

\vskip 3mm

\section{Representation Theory of $U_q(\gq(n))$}

Let us recall the highest weight  representation theory of $\Uq$
that was introduced in \cite{GJKK}.

\begin{definition} \hfill
\begin{enumerate}[{\rm (1)}]

\item A $U_q(\gq(n))$-module $M$ is a {\em weight module} if
it admits a weight space decomposition
$$    M= \bigoplus_{\mu \in P} M_{\mu},
       ~\text{where}~ M_{\mu} =\{ m \in M ~|~ q^{h}  m = q^{\mu(h)}m \ \text{ for all} \ h \in P^{\vee}\}. $$
\item A weight module $V$ is a {\em highest weight
module with highest weight $\lambda \in P$} if $V_{\la}$ is
finite-dimensional and satisfies the following conditions:
\begin{enumerate}[{\rm (i)}]
\item $V$ is generated by $V_{\la}$,
\item $e_i v = e_{\ol i} v =0$ for all $v \in V_{\la}$,
$i=1, \ldots, n-1$.
\end{enumerate}
\end{enumerate}
\end{definition}

For a weight $U_q(\qn)$-module $V$, we denote by $\wt(V)$ the set
of $\mu \in P$ such that $V_{\mu} \neq 0$. If $\dim_{\C((q))}
V_{\mu} < \infty$ for all $\mu \in P$, the {\em character} of $V$
is defined to be
  $$\ch V=\sum_{\mu \in P} (\dim_{\C((q))} V_{\mu})e^{\mu},  $$
where $e^{\mu}$ are formal basis elements of the group algebra
$\C[P]$ with the multiplication $e^{\la}e^{\mu}=e^{\la+\mu}$ for
all $\la, \mu \in P$.

As in the case of $\qn$, the Clifford superalgebra plays a central
role in the highest weight representation theory of $\Uq$. When
$m$ is a non-negative integer, the $q$-integer
$\dfrac{q^{2m}-q^{-2m}}{q^2-q^{-2}}$  has a square root in
$\C((q))$ but not in $\C(q)$. This difference gives the following
two statements, which is simpler than the corresponding statements
in \cite[Theorem 5.14]{GJKK}.

\begin{proposition} \label{prop:quantum
Clifford} For $\la \in P$, let ${\rm Cliff} _q(\la)$ be the
associative superalgebra over $\C((q))$ generated by odd
generators $\{t_{\ol i} \ | \ i=1,2,\ldots, n \}$ with the
defining relations
$$
\begin{array}{cc}
t_{\ol i} t_{\ol j} + t_{\ol j} t_{\ol i} = \delta_{ij}
\dfrac{2(q^{2\la_i}-q^{-2\la_i})}{q^2-q^{-2}}, & i,j =
1,2,\ldots,n.
\end{array}
$$
Then ${\rm Cliff} _q(\la)$ 
has up to isomorphism \begin{enumerate}[{\rm (1)}]
\item two irreducible modules $E^q(\la)$ and $\Pi(E^q(\la))$ of dimension $2^{k-1} | 2^{k-1}$ if $m=2k$,
\item one irreducible module $E^q(\la) \cong \Pi(E^q(\la))$ of dimension $2^k | 2^k$ if $m=2k+1$,
\end{enumerate}
where $m$ is the number of non-zero parts of $\la \in P$.
\end{proposition}

Let $U_q^{\geq 0}$ be the subalgebra of $U_q(\qn)$ generated by
$e_i, e_{\ib}$ ($i=1, \ldots, n-1$) and $q^h, k_{\ol j }$ $(h \in
P^{\vee}, j=1, \ldots, n)$.
In \cite{GJKK}, we proved the following proposition, which is a
quantum analogue of Proposition \ref{prop:penkov}.

\begin{proposition} \label{prop:penkov2} \cite[Proposition 4.1]{GJKK}
Let ${\bf v}^q$ be a finite-dimensional irreducible $U_q^{\geq 0
}$-module with a weight space decomposition.
\begin{enumerate}[{\rm (1)}]
\item The subalgebra $U_q^+$ of $U_q^{\geq 0}$ acts on ${\bf v}^q$ trivially.
\item There exists a unique weight $\lambda \in P$ such that ${\bf v}^q$
admits a ${\rm Cliff}_q(\lambda)$-module structure.
\item For all $h \in P^{\vee}, v \in {\bf
v}^q$, we have $q^h v =q^{\la(h)}v$.
\end{enumerate}
\end{proposition}

Combining Proposition \ref{prop:quantum Clifford} and Proposition
\ref{prop:penkov2}, we obtain a complete classification of
finite-dimensional irreducible weight $U_q^{\ge 0}$-modules. We
define
$$ W^q(\la) := U_q (\qn) \otimes_{U_q^{\geq
0}} E^q(\la) $$ to be the {\em Weyl module} of $U_q(\qn)$
corresponding to $\la$ (defined up to $\Pi$).

\vskip 3mm

 \begin{proposition} \cite[Proposition 4.2]{GJKK} \label{weyl}
  \begin{enumerate} [{\rm(1)}]

    \item $W^q(\la)$ is a free $U_q^-$-module of rank $\dim E^q(\la)$.

    \item Let $V$ be a highest weight $U_q(\qn)$-module with highest
    weight $\la$ such that $V_\la$ is an irreducible $U_q^{\ge 0}$-module.
    Then $V$ is a homomorphic image of $W^q(\la)$.

    \item Every Weyl module $W^q(\la)$ has a unique maximal submodule $N^q(\la)$.
  \end{enumerate}
 \end{proposition}

By Proposition \ref{weyl}, we see that there exists a unique
irreducible highest weight module $V^q(\la):=W^q(\la)/N^q(\la)$
with highest weight $\la \in P$ up to $\Pi$.

\vskip 3mm

\begin{example}
Consider the $\F$-vector space  $$\V = \soplus_{j=1}^n \F v_{j}
\oplus \soplus_{j=1}^n \F v_{\ol j}$$ with the action of
$U_q(\mathfrak{q}(n))$ given as follows:
\begin{equation}
\begin{array}{llll} e_iv_j=\delta_{j,i+1}v_i, &e_iv_{\ol
j}=\delta_{j,i+1}v_{\ol i},
&f_iv_j=\delta_{j,i}v_{i+1},&f_iv_{\ol j}=\delta_{j,i}v_{\ol{i+1}}, \\[1ex]
 e_{\ol i}v_j=\delta_{j,i+1}v_{\ol{i}},&e_{\ol i}v_{\ol j}=\delta_{j,i+1}v_{i},&
f_{\ol i}v_j=\delta_{j,i}v_{\ol{i+1}},&f_{\ol i}v_{\ol j}=\delta_{j,i}v_{{i+1}}, \\[1ex]
q^h v_j=q^{\epsilon_j(h)} v_j, &q^h v_{\ol j}=q^{\epsilon_j(h)}
v_{\ol j}, &k_{\ol i}v_j=\delta_{j,i}v_{\ol j},&k_{\ol i}v_{\ol
j}=\delta_{j,i}v_{j}. \end{array}
\end{equation}
Then $V$ is a $U_q(\qn)$-module and called the {\em vector
representation} of $U_q(\qn)$. Note that $\V$ is an irreducible
highest weight module with highest weight $\epsilon_1$.
 \end{example}

\vskip 3mm

Let
$$\Ao := \{f/g \in \C((q)) \ | \ f, g \in \C[[q]], \ g(1) \neq 0 \}$$
and let $V^q$ be a highest weight $U_q(\gq(n))$-module generated
by a finite-dimensional irreducible $U_q^{\geq 0}$-module $
E^q(\lambda)$. We denote by $\Cliff_{\Ao}(\la)$ the
$\Ao$-subalgebra of $\Cliff_q(\la)$ generated by $t_{\ol 1},
\ldots, t_{\ol n}$ and let $E^{\Ao}(\lambda)$ be  the
$\Cliff_{\Ao}(\lambda)$-submodule of $ E^q(\lambda)$ generated by
a nonzero even element in $E^q(\lambda)_{\bar0}$. The {\em
$\Ao$-form $U_{\Ao} $ of $U_q(\gq(n))$} is the $\Ao$-subalgebra of
$U_q(\gq(n))$ generated by $e_i,e_{\bar i},f_i, f_{\bar i}, q^h,
k_{\bar j}$ and $\dfrac{q^h-1}{q-1}$ for $i=1, \ldots, n-1, j=1,
\ldots,n$ and $h \in P^{\vee}$. The {\em $\Ao$-form $V^{\Ao}$ of
$V^q$ } is defined to be the $U_{\Ao}$-submodule of $V^q$
generated by $E^{\Ao}(\lambda)$.

Let $\Jo$ be the unique maximal ideal of $\Ao$ generated by $q-1$.
Then there is a canonical isomorphism of fields
\begin{equation*}
\Ao/\Jo \stackrel{\sim}{\longrightarrow} \C \ \ \ \text{given by }
\frac{f(q)}{g(q)}+\Jo \map \frac{f(1)}{g(1)}.
\end{equation*}
We define the {\em classical limit $U_1$ of $\Uq$} to be
$$\C\tensor_{\Ao} U_{\Ao} \cong U_{\Ao}/\Jo U_{\Ao}.$$
Similarly, the {\em classical limit $V^1$ of $V^q$} is defined to
be
$$\C\tensor_{\Ao} V_{\Ao} \cong  V^{\Ao}/\Jo V^{\Ao}.$$
The following {\it classical limit theorem} was proved in
\cite[Section 5]{GJKK}.

\begin{theorem} \cite[Theorem 5.11--Theorem 5.16]{GJKK}
\begin{enumerate}[{\rm (1)}]

\item As $U(\gq(n))$-modules, the classical limit $V^1$ of $V^q$ is isomorphic to a highest
weight $U(\qn)$-module $V$ with highest weight $\la \in P$ such
that $V_{\la}$ is an irreducible ${\mathfrak b}^{+}$-module.

\item $ \ch V^q = \ch V^1$.

\item The highest weight $\Uq$-module $V^{q}(\lambda)$ is finite dimensional if and only if $\lambda \in
\Lambda^{+}$.

\item If $V^q =V^q(\la)$ for $\la \in \Lambda^+ \cap P^{\geq 0}$, then
$V^1$ is isomorphic to $V(\la)$ up to $\Pi$.

\item The classical limit $U_1$ of $\Uq$ is isomorphic to $U(\gq(n))$ as $\C$-superalgebras.

\end{enumerate}
\end{theorem}

\begin{proof}
The assertion (1) can be verified by a direct calculation and the
assertion (2) follows from a couple of standard facts on tensor
products, in particular, on the extension of scalars of free
modules.

Combining Theorem \ref{thm:P}, Proposition \ref{prop:GJKK}, the
assertion (1) and (2), we obtain the assertion (3). Proposition
\ref{prop:GJKK} and the assertion (2) yield the assertion (4). Now
the assertion (5) can be proved as in \cite[Theorem 5.16]{GJKK}.

We would like to emphasize that the order of our assertions to be
proved is important and is carefully arranged.
\end{proof}

\vskip 3mm

We now introduce the main object of our investigation -- the
$\Uq$-modules in {\it the category $\mathcal{O}_{int}^{\ge 0}$}.

\begin{definition} \label{def:O}
 The {\em category
$\mathcal{O}_{int}^{\ge 0}$} consists of finite-dimensional
$U_q(\gq(n))$-modules $M$ with a weight space decomposition
satisfying the following conditions:

{\rm (1)} $\wt(M) \subset P^{\ge 0}$,

{\rm (2)} if $\lan k_i,\mu\ran=0$ for $\mu\in P^{\ge0}$ and $i \in
\{1,\ldots, n\}$, then $k_{\ol i}$ acts trivially on $M_{\mu}$.
\end{definition}

The fundamental properties of the category $\mathcal{O}_{int}^{\ge
0}$ are summarized in the following {\it complete reducibility
theorem}.

\begin{theorem} \cite[Proposition 6.2, Theorem 6.5]{GJKK}
\begin{enumerate}[{\rm (1)}]
\item Every $U_q(\gq(n))$-module in $\mathcal{O}_{int}^{\ge 0}$ is
completely reducible.
\item Every irreducible $U_q(\gq(n))$-module in ${\mathcal O}_{int}^{\ge 0}$ has the form $V^q(\lambda)$ for some
$\lambda \in  \La^+ \cap P^{\geq 0}$.
\end{enumerate}
\end{theorem}

\begin{proof}
Our assertions follow from the classical limit theorem and the
induction argument on the dimension of $\Uq$-modules in the
category $\mathcal{O}_{int}^{\ge 0}$. The condition (2) of
Definition \ref{def:O} plays a crucial role in the proof.
\end{proof}

\vskip 3mm

In the following theorem, we give a decomposition of the tensor
product of the vector representation with a highest weight
$\Uq$-module.

\Th \cite[Theorem 4.1(e)]{GJKKK1}, \cite[Theorem 1.11]{GJKKK2} \label{th:decomposition} Let $M$ 
be a  highest weight $\Uq$-module in $\Oint$ with highest weight
$\lambda \in \Lambda^+ \cap P^{\geq 0}$. Then we have
$$\V \otimes M \simeq\soplus_{\stackrel{\la + \epsilon_j :}{ \text{strict partition}}}
M_ j,$$ where $M_j$ is a highest weight $\Uq$-module in the
category $\Oint$ with highest weight $\la + \epsilon_j$ and $\dim
(M_{j})_{\la + \epsilon_j} = 2 \dim M_{\la}$. \enth

\begin{proof}
We first prove that our assertion holds for finite-dimensional
highest weight modules over $\qn$ in the category $\mathcal
O^{\geq 0}$. Then, by the classical limit theorem, our assertion
holds also for finite-dimensional highest weight modules in the
category $\Oint$.
\end{proof}

\begin{corollary} \cite[Corollary 1.12]{GJKKK2} \label{cor:Vtens}
Any irreducible $\Uq$-module in $\Oint$ appears as a direct
summand of tensor products of the vector representation $\V$.
\end{corollary}

\vskip 3mm

\section{Crystal Bases}

 Let $M$ be a $\Uq$-module in the category $\mathcal{O}_{int}^{\geq 0}$ and $I=\{ 1,2,
\ldots, n-1 \}$. For $i \in I$, we define the {\em even Kashiwara
operators} $\tilde{e_i}, \tilde{f_i}: M \ra M$ in the usual way.
That is, for $u \in M$, we write
 $$u=\sum_{k \geq 0} f_i^{(k)}u_k,$$ where $e_i u_k=0$ for
all $k \geq 0$ and $f_i^{(k)}=f_i^k/[k]!$, and we define
$$\tilde{e}_i u=\sum_{k \geq 1} f_i^{(k-1)}u_k, \
\tilde{f}_i u=\sum_{k \geq 0} f_i^{(k+1)}u_k.$$

On the other hand, we define the {\em odd Kashiwara operators} to
be
  \beq \bea
  &\tilde{k}_{\overline{1}} := q^{k_1-1}k_{\overline{1}},\\
  &\tilde{e}_{\overline{1}} := -(e_1k_{\ol{1}}-qk_{\ol 1}e_1)q^{k_1-1},\\
  &\tilde{f}_{\overline{1}} := -(k_{\ol{1}}f_1-qf_1 k_{\ol 1})q^{k_2-1}.
  \eea \eeq

Recall that an {\em abstract $\mathfrak{gl}(n)$-crystal} is a set
$B$ together with the maps $\tei, \tfi \colon B \to B \sqcup
\{0\}$, $\vphi_i, \varepsilon_i \colon B \to \Z \sqcup \{-\infty\}$
for $i \in I$, and $\wt\colon B \to P$ satisfying the following
conditions (see \cite{Kas93}):

\begin{enumerate}[{\rm (1)}]
\item $\wt(\tei b) = \wt (b) + \alpha_i$ if $i\in I$ and
$\tei b \neq 0$,

\item $\wt(\tfi b) = \wt (b) - \alpha_i$ if $i\in I$ and
$\tfi b \neq 0$,

\item for any $i \in I$ and $b\in B$, $\vphi_i(b) = \varepsilon_i(b) +
\langle h_i, \wt (b) \rangle$,

\item for any $i\in I$ and $b,b'\in B$,
$\tfi b = b'$ if and only if $b = \tei b'$,

\item for any $i \in I$ and $b\in B$
such that $\tei b \neq 0$, we have  $\varepsilon_i(\tei b) =
\varepsilon_i(b) - 1$, $\vphi_i(\tei b) = \vphi_i(b) + 1$,

\item for any $i \in I$ and $b\in B$ such that $\tfi b \neq 0$,
we have $\varepsilon_i(\tfi b) = \varepsilon_i(b) + 1$, $\vphi_i(\tfi b)
= \vphi_i(b) - 1$,

\item for any $i \in I$ and $b\in B$ such that $\vphi_i(b) = -\infty$, we
have $\tei b = \tfi b = 0$.
\end{enumerate}

In this paper, we say that an abstract $\mathfrak{gl}(n)$-crystal
is a {\em $\mathfrak{gl}(n)$-crystal} if it is realized as a
crystal basis of a finite-dimensional integrable
$U_q(\mathfrak{gl}(n))$-module. In particular, for any $b$ in a
$\mathfrak{gl}(n)$-crystal $B$, we have
$$\varepsilon_i(b)=\max\{n\in\Z_{\ge0}\,;\,\tei^nb\not=0\}, \quad
\vphi_i(b)=\max\{n\in\Z_{\ge0}\,;\,\tfi^nb\not=0\}.$$

\Def \label{def:crystal base} Let $M= \soplus_{\mu \in P^{\ge 0}}
M_{\mu}$ be a $U_q(\mathfrak{q}(n))$-module in the category
$\Oint$. A {\em crystal basis} of $M$ is a triple $(L, B,
l_{B}=(l_{b})_{b\in B})$, where \bna
\item $L$ is a free $\A$-submodule of $M$ such that

\bni
\item $\F \otimes_{\A} L \isoto M$,

\item $L = \soplus_{\mu \in P^{\ge 0}} L_{\mu}$, where $L_{\mu} = L
\cap M_{\mu}$,

\item  $L$ is stable under the Kashiwara operators $\tei$,
$\tfi$ $(i=1, \ldots, n-1)$, $\tkone$, $\teone$, $\tfone$.
\end{enumerate}

\item $B$ is a $\mathfrak{gl}(n)$-crystal together with
the maps $\teone, \tfone \colon B \to B \sqcup \{0\}$ such that

\bni
\item $\wt(\teone b) = \wt(b) + \alpha_1$, $\wt(\tfone b) = \wt(b) -
\alpha_1$,

\item for all $b, b' \in B$, $\tfone b = b'$ if and only if $b = \teone b'$.
\end{enumerate}

\item $l_{B}=(l_{b})_{b \in B}$ is a family of
non-zero subspaces of $L/qL$  such that

\bni
\item $l_{b} \subset (L/qL)_{\mu}$ for $b \in B_{\mu}$,

\item  $L/qL = \soplus_{b \in B} l_{b}$,

\item $\tkone l_{b} \subset l_{b}$,
\item for $i=1, \ldots, n-1, \ol 1$, we have
\begin{enumerate}[{\rm(1)}]
\item
if $\tei b=0$ then $\tei l_{b} =0$, and otherwise $\tei$ induces
an isomorphism $l_{b}\isoto l_{\tei b}$,
\item
if $\tfi b=0$ then $\tfi l_{b} =0$, and otherwise $\tfi$ induces
an isomorphism $l_{b}\isoto l_{\tfi b}$. \end{enumerate}
\end{enumerate}
\end{enumerate}

\edf

\begin{remark}  Note that an element $b \in B$ does {\it not}
correspond to a basis vector of $L / qL$. Instead, it corresponds
to a subspace $l_{b}$ of $L / qL$. In \cite[Proposition
2.3]{GJKKK2}, we proved that for any crystal basis $(L, B, l_B)$
of a $\Uq$-module $M \in \Oint$, we have $\teone^2 = \tfone^2 = 0$
as endomorphisms on $L/qL$.
\end{remark}

\begin{example}
Let $\V = \soplus_{j=1}^n \F v_{j} \oplus \soplus_{j=1}^n \F
v_{\ol j} $ be the vector representation of $U_q(\qn)$. Set
$$\mathbf{L} = \soplus_{j=1}^n \A v_{j} \oplus \soplus_{j=1}^n \A
v_{\ol j}\quad \text{and }l_{j} = \C v_{j} \oplus \C v_{\ol j}
\subset \mathbf{L}/ q \mathbf{L},$$ and let $\B$ be the
$\mathfrak{gl}(n)$-crystal with the $\bar 1$-arrow given below.

\begin{equation*}
 \B \ : \ \xymatrix@C=5ex {*+{\young(1)} \ar@<0.1ex>[r]^-{1}
\ar@{-->}@<-0.9ex>[r]_{\ol 1} & *+{\young(2)} \ar[r]^2 &
*+{\young(3)} \ar[r]^3 & \cdots \ar[r]^{n-1} & *+{\young(n)} }
\end{equation*}
Here, the actions of $\tfi$ $(i=1, \ldots, n-1, \ol 1)$ are
expressed by $i$-arrows. Then $(\mathbf{L}, \B,
l_{\B}=(l_j)_{j=1}^n)$ is a crystal basis of $\V$.
\end{example}

\begin{remark}
Let $M$ be a $\uqqn$-module in the category  $\Oint$ with a
crystal basis $(L,B,l_{B})$. For $i=1,\ldots,n-1,\ob$ and $b$,
$b'\in B$, if $b'=\tf_ib$, then we have isomorphisms $\tf_i \colon
l_b \isoto l_{b'}$ and $\te_i \colon l_{b'} \isoto l_b$. If
$i=1,\ldots,n-1$, then they are inverses to each other. However,
when $i=\ol 1$, they are not inverses to each other in general.
\end{remark}

The {\it queer tensor product rule} given in the following theorem
is one of the most important and interesting features of the
crystal basis theory of $\Uq$-modules.

\Th \cite[Theorem 3.3]{GJKKK1} \cite[Theorem 2.7]{GJKKK2}
\label{th2:tensor product}
 Let $M_j$ be a $\uqqn$-module in $\Oint$ with a crystal basis
$(L_j, B_j, l_{B_j})$ $(j=1,2)$. Set $B_1\otimes B_2 = B_1 \times
B_2$ and $l_{b_1\otimes b_2}=l_{b_1} \otimes l_{b_2}$ for $b_1\in
B_1$ and $b_2\in B_2$. Then $(L_1 \otimes_{\A} L_2, B_1 \otimes
B_2,(l_b)_{b\in B_1 \otimes B_2})$ is a crystal basis of $M_1
\otimes_{\F}M_2$, where the action of the Kashiwara operators on
$B_1\otimes B_2$ are given as follows:
\eq &&\begin{aligned} \tei(b_1 \otimes b_2) & = \begin{cases} \tei
b_1 \otimes b_2 \ &
\text{if} \ \vphi_i(b_1) \ge \eps_i(b_2), \\
b_1 \otimes \tei b_2 \ & \text{if} \ \vphi_i(b_1) < \eps_i(b_2),
\end{cases} \\
\tfi(b_1 \otimes b_2) & = \begin{cases} \tfi b_1 \otimes b_2 \
& \text{if} \  \vphi_i(b_1) > \eps_i(b_2), \\
b_1 \otimes \tfi b_2 \ & \text{if} \ \vphi_i(b_1) \le \eps_i(b_2),
\end{cases}
\end{aligned}\\[2ex]
 \label{eq1:tensor product}
&&\begin{aligned} \teone (b_1 \otimes b_2) & = \begin{cases}
\teone b_1 \otimes b_2
& \text{if} \ \lan k_1, \wt (b_2) \ran =   \lan k_2, \wt (b_2) \ran =0, \\
b_1 \otimes \teone b_2   
&  \text{otherwise,}
\end{cases} \\
\tfone(b_1 \otimes b_2) & = \begin{cases} \tfone b_1 \otimes b_2
& \text{if} \ \lan k_1, \wt (b_2) \ran = \lan k_2, \wt (b_2) \ran =0, \\
b_1 \otimes \tfone b_2   
& \text{otherwise}.
\end{cases}\\
\end{aligned}
\label{eq2:tensor product} \eneq
 \enth

\begin{proof}
For $i=1, 2, \ldots, n-1$, our assertions were already proved in
\cite{Kas90, Kas91}. For $i= {\ol 1}$, our assertions follow from
the following comultiplication formulas (see \cite{GJKKK2}):
\begin{equation*}
\begin{aligned}
& \Delta(\tkone) = \tkone \otimes q^{2 k_1} + 1 \otimes \tkone, \\
& \Delta(\teone) = \teone \otimes q^{k_1 + k_2} + 1 \otimes \teone
 - (1-q^2) \tkone \otimes e_1 q^{2 k_1}, \\
& \Delta(\tfone) = \tfone \otimes q^{k_1 + k_2} + 1 \otimes \tfone
 - (1-q^{2}) \tkone \otimes f_1 q^{k_1 + k_2-1}.
\end{aligned}
\end{equation*}
\end{proof}

\Def An {\em abstract $\mathfrak{q}(n)$-crystal} is a
$\mathfrak{gl}(n)$-crystal together with the maps $\teone,
\tfone\colon B \to B \sqcup \{0\}$ satisfying the following
conditions: \bna
\item $\wt(B)\subset P^{\ge0}$,
\item $\wt(\teone b) = \wt(b) + \alpha_1$, $\wt(\tfone b) = \wt(b) -
\alpha_1$,

\item for all $b, b' \in B$, $\tfone b = b'$ if and only if $b = \teone b'$,

\item 
if $3\le i\le n-1$, we have \begin{enumerate}[{\rm(i)}]
\item
the operators $\teone$ and $\tfone$ commute with $\te_i$ and
$\tf_i$ ,

\item if $\teone b\in B$, then
$\eps_i(\teone b)=\eps_i(b)$ and $\vphi_i(\teone b)=\vphi_i(b)$.
\end{enumerate}
\end{enumerate}
 \edf

Let $B_1$ and $B_2$ be abstract $\qn$-crystals. The {\em tensor
product} $B_1 \otimes B_2$ of $B_1$ and $B_2$ is defined to be the
$\mathfrak{gl}(n)$-crystal $B_1 \otimes B_2$ together with the
maps $\teone$, $\tfone$ defined by \eqref{eq2:tensor product}.
Then it is an abstract $\qn$-crystal.
\begin{remark}
Let $B_1, B_2$ and $B_3$ be abstract $\qn$-crystals. Then we have
$$(B_1 \otimes B_2) \otimes B_3 \simeq B_1\otimes (B_2 \otimes B_3).$$
\end{remark}

\begin{example} \hfill \label{ex_abstract_crystal}
\bna
\item If $(L, B, l_{B})$ is a crystal basis of a $\Uq$-module $M$ in the
category $\Oint$, then $B$ is an abstract $\qn$-crystal.

\item The crystal graph $\B$ of the vector representation $\V$
is an abstract $\qn$-crystal.

\item By the tensor product rule, $\B^{\otimes N}$ is an abstract
$\qn$-crystal. When $n=3$, the $\qn$-crystal structure of $\B
\otimes \B$ is given below.

$$
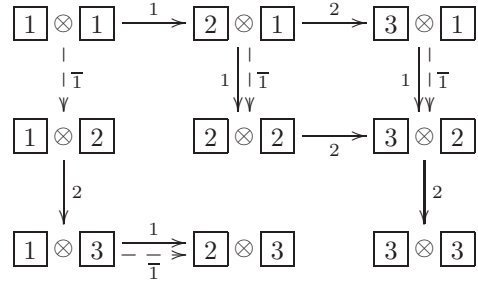
\xymatrix
{*+{\young(1) \otimes \young(1)} \ar[r]^1 \ar@{-->}[d]^{\ol 1} &
 *+{\young(2) \otimes \young(1)} \ar@<-0.5ex>[d]_1 \ar@{-->}@<0.5ex>[d]^{\ol 1} \ar[r]^2&
 *+{\young(3) \otimes \young(1)} \ar@<-0.5ex>[d]_1 \ar@{-->}@<0.5ex>[d]^{\ol 1} \\
 *+{\young(1) \otimes \young(2)} \ar[d]^2 &
 *+{\young(2) \otimes \young(2)} \ar[r]_2 &
 *+{\young(3) \otimes \young(2)} \ar[d]^2 \\
 *+{\young(1) \otimes \young(3)} \ar@{-->}@<-0.5ex>[r]_{\ol 1} \ar@<0.5ex>[r]^{1} &
 *+{\young(2) \otimes \young(3)} &
 *+{\young(3) \otimes \young(3)}
 }$$

\item For a strict partition $\la = (\la_1 > \la_2 > \cdots > \la_r
>0)$, 
let $Y_{\la}$ be the skew Young diagram having $\la_1$ many boxes
in the principal diagonal, $\la_2$ many boxes in the second
diagonal, etc.
For example, if $\la=(7 > 6 > 4 > 2 > 0)$, then we have

$$Y_{\la} = \young(::::::\hfill,:::::\hfill\hfill,::::\hfill\hfill\hfill,:::\hfill\hfill\hfill\hfill,::\hfill\hfill\hfill\hfill,:\hfill\hfill\hfill,\hfill\hfill) \quad.$$

Let $\B(Y_{\la})$ be the set of all semistandard tableaux of shape
$Y_{\la}$ with entries from $1, 2, \ldots, n$. Then by an {\em
admissible reading} introduced in \cite{BKK}, $\B(Y_{\la})$ can be
embedded in $\B^{\otimes N}$, where $N=\la_1 + \cdots + \la_r$.
One can show that it is stable under the Kashiwara operators
$\tei,\tfi$ ($i=1, \cdots, n-1,\ol1$) and hence it becomes an
abstract $\qn$-crystal. Moreover, the $\qn$-crystal structure thus
obtained does not depend on the choice of admissible reading.

In Figure \ref{fi:B(Y_(3,1,0))}, we illustrate the crystal
$\B(Y_\la)$ for $n=3$ and $\la=(3>1>0)$. In Figure
\ref{fi:B(Y_(3,0,0))}, we present the crystal $\B(Y_{\mu})$ for
$n=3$ and $\mu=(3>0)$. Note that in general, $\B(Y_\la)$ is not
connected.
   \ena
\end{example}

\begin{figure}[h]
 $$\scalebox{.8}{\xymatrix@R=1pc@H=1pc{ & &  {\young(::1,:12,1)} \ar[dl]_1 \ar_2[d]  \ar^{\ol 1}@{-->}[dr] & & & \\
& {\young(::1,:22,1)} \ar@<-0.5ex>_1[dl] \ar@<1ex>^{\ol
1}@{-->}[dl] \ar_2[d] & {\young(::1,:13,1)} \ar_1[d]\ar^{\ol
1}@{-->}[dr] &
{\young(::1,:12,2)} \ar_2[d] & & \\
{\young(::1,:22,2)} \ar_2[d] & {\young(::1,:23,1)}\ar@<-0.5ex>_1
[dl] \ar@<1ex>^{\ol 1}@{-->}[dl] \ar_2[d] & {\young(::2,:13,1)}
\ar_1 [d] \ar^{\ol 1}@{-->}[dr] & {\young(::1,:13,2)} \ar_1 [d]
\ar_2[dr] & & {\young(::1,:12,3)} \ar@<-0.5ex>_1 [d]
\ar@<1ex>^{\ol 1}@{-->}[d] \\
 {\young(::1,:23,2)} \ar_2[d] & {\young(::1,:33,1)} \ar^{\ol 1}@{-->}[dl] \ar_1 [d] & {\young(::2,:23,1)} \ar@<-0.5ex>_1 [d]
\ar@<1ex>^{\ol 1}@{-->}[d] \ar_2[dl] & {\young(::2,:13,2)}
 \ar_2[d] & {\young(::1,:13,3)} \ar_1 [dl] \ar^{\ol 1}@{-->}[dr] & {\young(::1,:22,3)} \ar_2[d] \\
 {\young(::1,:33,2)} \ar_2[dr] & {\young(::2,:33,1)} \ar@<-0.5ex>@{-->}_{\ol 1}[dr] \ar@<1ex>^{1}[dr]  &{\young(::2,:23,2)} \ar_2[d] & {\young(::2,:13,3)} \ar@<-0.5ex>_1 [d] \ar@<1ex>^{\ol 1}@{-->}[d] & & {\young(::1,:23,3)} \\
   & {\young(::1,:33,3)} \ar@<-0.5ex>@{-->}_{\ol 1}[dr] \ar@<1ex>^{1}[dr]       &  {\young(::2,:33,2)} \ar_2[d] &   {\young(::2,:23,3)} &  &  \\
   & &{ \young(::2,:33,3)}&& &}}$$
\caption{${\mathbf B}(Y_\la)$ for $n=3$, $\la = (3>1>0)$.}
\label{fi:B(Y_(3,1,0))}
   \end{figure}

\begin{figure}[h]
 $$\scalebox{.8}{\xymatrix@R=1pc@H=1pc{ &  {\young(::1,:1,1)} \ar_1 [d] \ar^{\ol 1}@{-->}[dr] &&&  &&  \\
& {\young(::2,:1,1)} \ar_1[dl] \ar_2[d]  \ar^{\ol 1}@{-->}[dr]  & {\young(::1,:1,2)}   \ar_1[d] \ar_2[dr]  & &   & {\young(::1,:2,1)} \ar@<-0.5ex>_1[d] \ar@<1ex>^{\ol 1}@{-->}[d] \ar[dr]_2  &  \\
 {\young(::2,:2,1)} \ar@<-0.5ex>_1 [d] \ar@<1ex>^{\ol 1}@{-->}[d] \ar[dr]_2 & {\young(::3,:1,1)} \ar_1[d] \ar^{\ol 1}@{-->}[dr]   & {\young(::2,:1,2)} \ar_2[d]& {\young(::1,:1,3)} \ar[d]_1 \ar^{\ol 1}@{-->}[dr]& & {\young(::1,:2,2)} \ar[d]_2 &  {\young(::1,:3,1)} \ar_{\ol 1}@{-->}[dl] \ar[d]_1 \\
  {\young(::2,:2,2)}\ar[d]_2 & {\young(::3,:2,1)}  \ar@<-0.5ex>_1
[dl] \ar@<1ex>^{\ol 1}@{-->}[dl]  \ar[d]_2  & {\young(::3,:1,2)}  \ar[d]_2 & {\young(::2,:1,3)} \ar@<-0.5ex>_1 [d] \ar@<1ex>^{\ol 1}@{-->}[d] & {\young(::1,:2,3)}  & {\young(::1,:3,2)} \ar[d]_2 &  {\young(::2,:3,1)} \ar@<-0.5ex>_1[d] \ar@<1ex>^{\ol 1}@{-->}[d]  \\
{\young(::3,:2,2)} \ar[dr]_2 & {\young(::3,:3,1)} \ar@<-0.5ex>_1[d] \ar@<1ex>^{\ol 1}@{-->}[d]  & {\young(::3,:1,3)} \ar@<-0.5ex>_1 [d] \ar@<1ex>^{\ol 1}@{-->}[d] & {\young(::2,:2,3)}\ar[dl]^2 &  &{\young(::1,:3,3)} \ar@<-0.5ex>_1[d] \ar@<1ex>^{\ol 1}@{-->}[d] &  {\young(::2,:3,2)} \ar[dl]_2 \\
& {\young(::3,:3,2)}\ar[d]_2      & {\young(::3,:2,3)} & & &{\young(::2,:3,3)}  &\\
&{\young(::3,:3,3)} &&& &&}}$$
 \caption{${\mathbf B}(Y_\mu)$ for $n=3$, $\mu = (3>0)$.}
\label{fi:B(Y_(3,0,0))}
   \end{figure}
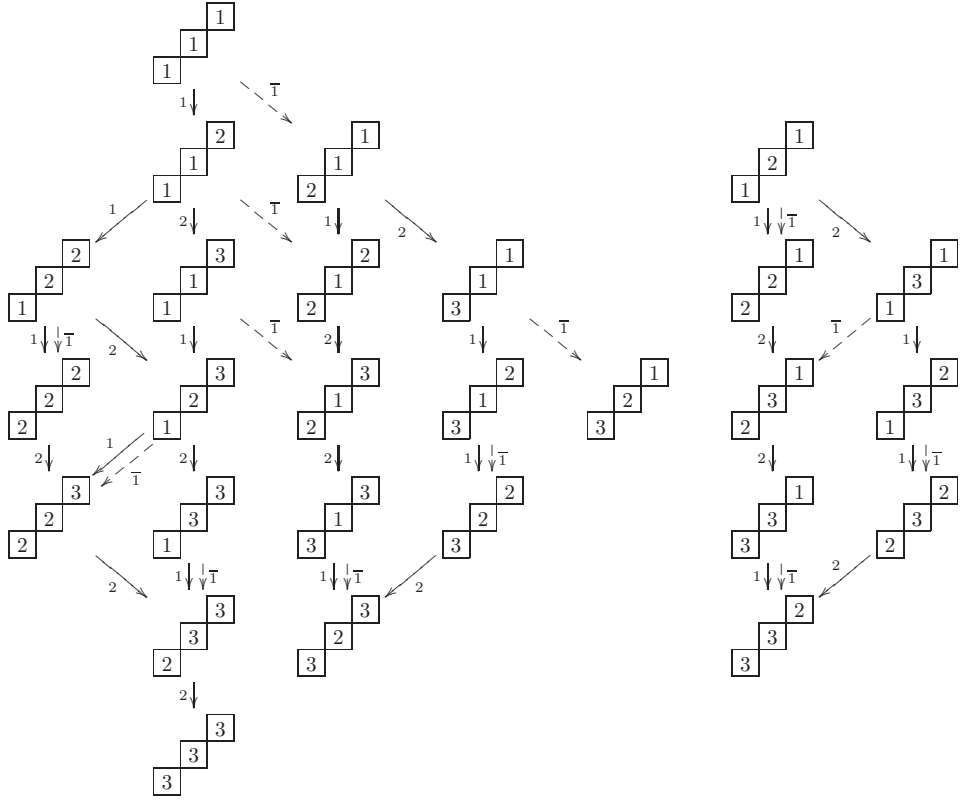

Let $B$ be an abstract $\qn$-crystal. For $i = 1,2,\ldots, n-1$,
we define the automorphism $S_i$ on $B$ by \eq &&S_i b =
\begin{cases}
\tfi^{\langle h_i, \wt b \rangle} b & \text{if} \quad {\langle h_i, \wt b \rangle} \geq 0, \\
\tei^{-\langle h_i, \wt b \rangle} b & \text{if} \quad {\langle
h_i, \wt b \rangle} \leq 0.
\end{cases}\label{def:Sicr}
\eneq Let $w$ be an element of the Weyl group $W$ of
$\mathfrak{gl}(n)$. Then, as shown in \cite{Kas94}, there exists a
unique action $S_{w} \colon B \to B$ of $W$ on $B$ such that
$S_{s_i}=S_i$ for $i=1,2,\ldots,n-1$. Note that $\wt(S_w
b)=w(\wt(b))$ for any $w \in W$ and $b \in B$.

For $i=1, \ldots, n-1$, we set \eq &&w_i = s_2 \cdots s_{i} s_1
\cdots s_{i-1}. \label{def:wi} \eneq Then $w_i$ is the shortest
element in $W$ such that $w_i(\alpha_i) = \alpha_1$. We define the
{\em odd Kashiwara operators} $\teibar$, $\tfibar$ $(i=2, \ldots,
n-1)$ by
$$\teibar = S_{w_i^{-1}} \teone
S_{w_i}, \ \ \tfibar = S_{w_i^{-1}} \tfone S_{w_i}.$$ We say that
an element $b \in B$ is a {\em highest weight vector} if $\tei b =
\teibar b =0$ for all $i=1, \ldots, n-1$, and an element $b \in B$
is a \emph{lowest weight vector} if $S_{w_0}b$ is a highest weight
vector, where $w_0$ is the longest element of $W$.



\vskip 3mm

In the following lemma, we give a combinatorial characterization
of highest weight vectors in $\B^{\tensor N}$, which plays a
crucial role in the proof of the main theorem. We expect this
lemma will have many important applications in the combinatorial
representation theory of $\Uq$.

\begin{lemma} \cite[Theorem 3.11]{GJKKK2} \label{le:char.h.w}
A vector $b_{0} \in \B^{\otimes N}$ is a highest weight vector if
and only if $b_{0} = 1 \otimes  \tf_1 \cdots \tf_{j-1} b$ for some
$j \in \{1, 2, \ldots, n \}$ and some highest weight vector $b \in
\B^{\otimes (N-1)}$ such that $\wt(b_0) = \wt(b)+\epsilon_j$ is a
strict partition.
\end{lemma}

\begin{proof}
The proof consists of series of lemmas and lengthy (and careful)
case-by-case check-ups.
\end{proof}

\vskip 3mm

The existence and the uniqueness of crystal bases of $\Uq$-modules
in $\Oint$ is given in the following theorem.

\Th \cite[Theorem 4.1]{GJKKK1} \cite[Theorem 4.6]{GJKKK2}
\label{th:main theorem} \hfill \bna
\item  Let $M$ be an irreducible highest weight $\Uq$-module with highest weight
$\la \in \Lambda^+ \cap P^{\geq 0}$. Then there exists a crystal
basis $(L, B, l_{B})$ of $M$ such that

\bni

\item $B_{\la} = \{b_{\la} \}$,

\item $B$ is connected.

\end{enumerate}
Moreover, such a crystal basis is unique up to an automorphism of
$M$. In particular, $B$ depends only on $\la$ as an abstract
$\qn$-crystal and we write $B=B(\la)$.

\item The $\qn$-crystal $B(\la)$ has a unique highest weight vector
$b_{\la}$ and a unique lowest weight vector $l_{\la}$. 

\item A vector $b \in \B \otimes B(\la)$ is a highest weight vector
if and only if
$$b = 1 \otimes \tf_{1} \cdots \tf_{j-1} b_{\la}$$
for some $j$ such that $\la + \epsilon_j$ is a strict partition.

\item Let $M$ be a finite-dimensional highest weight $\Uq$-module
with highest weight
$\la \in \Lambda^+ \cap P^{\geq 0}$. 
Assume that $M$ has a crystal basis $(L,B(\la), l_{B(\la)})$ such
that $L_{\la} / q L_{\la} = l_{b_{\la}}$. Then we have

\bni

\item $\V \otimes M = \soplus_{\la + \epsilon_j : \text{strict}}M_ j,$
where $M_j$ is a highest weight $\Uq$-module with highest weight
$\la + \epsilon_j$ and $\dim (M_{j})_{\la + \epsilon_j} = 2\dim
M_{\la}$,

\item 
if we set $L_{j} = (\mathbf{L} \otimes L) \cap M_{j}$ and $B_j=\{
b \in \B \otimes B(\la) \ | \ l_b\subset L_j/qL_j \}$, then we
have $\B \otimes B(\la)=\coprod\limits_{\la+ \epsilon_j:
\text{strict}} B_{j}$ and $L_j/qL_j = \soplus_{b \in B_j} l_{b}$,
\item $M_j$ has a crystal basis $(L_j, B_j, l_{B_j})$,

\item $B_{j}\simeq B(\la + \epsilon_j)$ as an abstract $\qn$-crystal.

\end{enumerate}
\end{enumerate}
 \enth

\begin{proof}
All of these assertions are proved by a series of interlocking
inductive arguments (see \cite{GJKKK2}).
\end{proof}

\vskip 3mm

Our main theorem implies the following corollary.
\begin{corollary} \cite[Theorem 4.1(d)]{GJKKK1} \cite[Corollary 4.7]{GJKKK3} \hfill
\bna
\item Every $U_q(\mathfrak{q}(n))$-module in the category
$\Oint$ has a crystal basis.
\item If $M$ is a $\Uq$-module in the category
$\Oint$ and $(L,B,l_B)$ is a crystal basis of $M$, then there
exist decompositions $M=\soplus_{a\in A}M_a$ as a
$U_q(\mathfrak{q}(n))$-module, $L=\soplus_{a\in A}L_a$ as an
$\A$-module, $B=\coprod_{a\in A}B_a$ as a $\qn$-crystal,
parametrized by a set $A$ such that the following conditions are
satisfied for any $a\in A$ : \bni
\item
$M_a$ is a highest weight module with highest weight $\la_a$ and
$B_a \simeq B(\la_a)$ for some strict partition $\la_a$,
\item
$L_a=L\cap M_a$, $L_a/qL_a=\soplus_{b\in B_a}l_b$,
\item $(L_a,B_a,l_{B_a})$ is a crystal basis of $M_a$.
\end{enumerate}
\ena
\end{corollary}

\vskip3mm
\section{Semistandard decomposition tableaux}

As we have seen in Example \ref{ex_abstract_crystal} (4), the
abstract $\gq(n)$-crystal $\B(Y_{\la})$ is usually too big to be
isomorphic to $B(\la)$, the crystal of the irreducible highest
weight module $V^q(\la)$. In this section, we give an explicit
combinatorial realization of the $\qn$-crystal $B(\la)$ in terms
of {\em semistandard decomposition tableaux}.

\begin{definition} (cf. \cite[Section 1.2]{Serra}) \hfill
\begin{enumerate}[{\rm (1)}]

\item A word $u = u_1 \cdots  u_N$  is a {\em hook word} if
there exists $1 \le k \le N$ such that
\begin{equation*}\label{hookequation}
u_1 \ge u_2 \ge \cdots \ge u_k <u_{k+1} <\cdots < u_N.
\end{equation*}
 Every hook word has the \emph{decreasing part} $u \downarrow   = u_1 \cdots u_k$,
 and the \emph{increasing part} $u \uparrow = u_{k+1} \cdots u_N$ (note that the decreasing part is always nonempty).

\item For a strict partition $\lambda = (\lambda_1,\ldots,\lambda_n)$, the {\em shifted Young diagram of
shape $\lambda$} is an array of boxes in which the $i$-th
row has $\lambda_i$ many boxes, and is shifted $i-1$ units to the right
with respect to the top row. In this case, we  say that $\lambda$
is a \emph{shifted shape}.
\end{enumerate}
\end{definition}

\begin{example}
For $\la=(6,4,2,1)$, the shifted shape $\la$ is
$$\young(\hfill\hfill\hfill\hfill\hfill\hfill,:\hfill\hfill\hfill\hfill,::\hfill\hfill,:::\hfill) .$$
\end{example}

\begin{definition} (cf. \cite[Definition 2.14]{Serra}) \hfill
\begin{enumerate}[{\rm (1)}] \label{def_ssdt}
\item A {\em semistandard decomposition tableau of a
shifted shape $\lambda = (\lambda_1, \ldots, \lambda_n)$} is a
filling $T$ of a shifted shape $\lambda$ with elements of
$\{1,2,\ldots,n\}$ such that:
\begin{enumerate}
\item[(i)] the word $v_i$ formed by reading the $i$-th row from left to right is a hook word of length
$\lambda_i$,
\item[(ii)] $v_i$ is a hook subword of maximal length in $v_{i+1} v_i$ for $1 \le i \le \ell(\lambda)-1$. 
\end{enumerate}
\item The {\em reading word} of a semistandard decomposition tableau $T$ is
$${\rm read}(T) = v_{\ell(\lambda)} v_{\ell(\lambda)-1} \cdots v_1.$$
\end{enumerate}
\end{definition}

\begin{remark}
We change the definition of a hook word, and hence of a
semistandard decomposition tableau in \cite{Serra}, in order to
make the forms of the highest weight vectors and the lowest weight
vectors simpler than the ones in \cite{Serra}.
\end{remark}



\begin{example}
The following tableaux are semistandard decomposition tableaux
of a shifted shape $(3,1,0)$:
$$ {\young(211,:1)} \ ,   \ \   {\young(222,:1)} \ , \ \
{\young(213,:1)} \ ,    \ \ {\young(212,:1)}\ .$$

\noindent On the other hand, the following tableaux do not satisfy
the conditions in Definition \ref{def_ssdt} (1):
$$ {\young(121,:1)}\ , \ \  {\young(123,:1)} \ .$$
\end{example}

Let $\B (\lambda)$ be the set of all semistandard decomposition
tableaux $T$ with a shifted shape $\lambda$. For every strict
partition $\lambda$, we have the embedding
$$\read: \B (\lambda) \to \B^{\otimes |\lambda|}, \; T \mapsto
\read(T),$$ which enables us to identify $\B(\lambda)$ with a
subset in $\B^{\otimes |\lambda|}$ and define the action of the
Kashiwara operators $\tei, \teibar, \tfi, \tfibar$ on $\B(\la)$ by
the queer tensor product rule.

\begin{theorem}\cite[Theorem 2.5]{GJKKK3} The set $ \B (\lambda) \cup \{0\}$ is
stable under the Kashiwara operators $\tei, \teibar, \tfi,
\tfibar$. Hence, $\B (\lambda)$ becomes an abstract
$\mathfrak{q}(n)$-crystal.
\end{theorem}

\begin{proof}
We first show that if $u$ is a hook word, then $\te_i u, \tf_i u$
($i=1, \ldots,n-1, \ol 1$) are hook words whenever they are nonzero.
Next, we prove that $\tf_i u, \te_i u$ ($i=1,...,n-1, \ol{1}$)
satisfy the condition in Definition \ref{def_ssdt} (1)(ii). For
this, we show that if $\tf_i u, \te_i u$ ($i=1,...,n-1, \ol{1}$) has
a hook subword of length $m$, then $u$ also has a hook subword of
length $m$ when $u \in \B(\la)$ and $\la_3=0$. The proof is based on
case-by-case check-ups.
\end{proof}


For a strict partition $\lambda$ with $\ell(\lambda) =r$, set
\begin{align}
T^{\lambda}:=& (1^{\lambda_r}) (2^{\lambda_r} 1^{\lambda_{r-1} -
\lambda_r}) \cdots ((r-k+1)^{\lambda_r}
(r-k)^{\lambda_{r-1}-\lambda_r} \cdots \nonumber 1^{\lambda_{k} - \lambda_{k+1}}) \\
&   \cdots (r^{\lambda_r} (r-1)^{\lambda_{r-1}-\lambda_r} \cdots
1^{\lambda_1 - \lambda_2}), \nonumber \\
L^{\lambda}:=&(n-r+1)^{\lambda_r} \cdots (n-k+1)^{\lambda_k}
\cdots n^{\lambda_1}. \nonumber
\end{align}
It is easy to check that $S_{w_0}T^{\la}=L^{\la}$.

\begin{example} \label{ex_hw}
Let  $n=4$ and $\lambda = (7,4,2,0)$. Then we have
$$T^{\lambda}=\young(3322111,:2211,::11) \ \text{and} \ L^{\lambda}=\young(4444444,:3333,::22).$$
\end{example}

The explicit combinatorial realization of $B(\la)$ is given in the
following lemma.

\begin{theorem} \cite[Theorem 2.5]{GJKKK3} Let $\la$ be a strict partition.
\bna
\item
The tableau $T^{\lambda}$ is a unique highest weight vector in $\B
(\lambda)$ and $L^{\lambda}$ is a unique lowest weight vector in
$\B (\lambda)$.
\item The abstract $\qn$-crystal $\B(\la)$ is isomorphic to $B(\la)$, the crystal of the
irreducible highest weight module $V^q(\la)$. \end{enumerate}
\end{theorem}

\begin{proof}
Using Lemma \ref{le:char.h.w}, the lowest weight vectors are
characterized as follows:
\begin{equation} \label{ch:lowest wt vector}
\text{\parbox{10cm}{For $a \in \B$ and $b \in
\B^{\otimes N}$, $a \otimes b$ is a lowest weight vector if and
only if $b$ is a lowest weight vector and $\epsilon_a + \wt(b) \in w_0 (\Lambda^+ \cap P^{\ge 0})$.}}
\end{equation}
Using induction on $|\la|$ and the above statement, we conclude
that $L^{\la}$ is a unique lowest weight vector in $\B(\la)$.
Since $S_{w_0}T^{\la}=L^{\la}$, we get the first assertion. The
second assertion follows from the first one directly.
\end{proof}

\vskip 3mm

\begin{example} \hfill
\bna
\item Since any word of length 2 is a hook word, we
obtain $\B \otimes \B \simeq \B(2\epsilon_1)$, and hence the
crystal in Example \ref{ex_abstract_crystal} (3) is isomorphic to
the $\gq(3)$-crystal $\B(2\epsilon_1)$.

\item In Figure~\ref{fi:B(3,1,0)}, we present the
$\gq(3)$-crystal $\B(3\epsilon_1 +\epsilon_2)$.

 \begin{figure}[h]
$$\scalebox{.7}{\xymatrix@R=1pc@H=1pc{ & &  {\young(211,:1)} \ar[dl]_1 \ar_2[d]  \ar^{\ol 1}@{-->}[dr] & & & \\
& {\young(221,:1)} \ar@<-0.5ex>_1[dl] \ar@<1ex>^{\ol 1}@{-->}[dl]
\ar_2[d] & {\young(311,:1)} \ar_1[d]\ar^{\ol 1}@{-->}[dr] &
{\young(212,:1)} \ar_2[d] & & \\
{\young(222,:1)} \ar_2[d] & {\young(321,:1)}\ar@<-0.5ex>_1 [dl]
\ar@<1ex>^{\ol 1}@{-->}[dl] \ar_2[d] & {\young(311,:2)}  \ar_1 [d]
\ar^{\ol 1}@{-->}[dr] & {\young(312,:1)} \ar_1 [d] \ar_2[dr] & &
{\young(213,:1)} \ar@<-0.5ex>_1 [d]
\ar@<1ex>^{\ol 1}@{-->}[d] \\
 {\young(322,:1)} \ar_2[d] & {\young(331,:1)} \ar^{\ol 1}@{-->}[dl] \ar_1 [d] & {\young(321,:2)} \ar@<-0.5ex>_1 [d]
\ar@<1ex>^{\ol 1}@{-->}[d] \ar_2[dl] & {\young(312,:2)}
 \ar_2[d] & {\young(313,:1)} \ar_1 [dl] \ar^{\ol 1}@{-->}[dr] & {\young(223,:1)} \ar_2[d] \\
 {\young(332,:1)} \ar_2[dr] & {\young(331,:2)} \ar@<-0.5ex>@{-->}_{\ol 1}[dr] \ar@<1ex>^{1}[dr]
 &{\young(322,:2)} \ar_2[d] & {\young(313,:2)} \ar@<-0.5ex>_1 [d] \ar@<1ex>^{\ol 1}@{-->}[d] &
 & {\young(323,:1)} \\
   & {\young(333,:1)} \ar@<-0.5ex>@{-->}_{\ol 1}[dr] \ar@<1ex>^{1}[dr]    &  {\young(332,:2)} \ar_2[d] &   {\young(323,:2)} &  &  \\
   & &{ \young(333,:2)}&& &}}$$
\caption{$\B(3\epsilon_1 +\epsilon_2)$ for $n=3$}
\label{fi:B(3,1,0)}
\end{figure}

\item In Figure \ref{fi:B(3,0,0)} and Figure \ref{fi:B(2,1,0)}, we
illustrate the $\gq(3)$-crystal $\B(3 \epsilon_1)$ and $\B(2
\epsilon_1 + \epsilon_2)$, respectively.
By Lemma \ref{le:char.h.w} there are two highest weight vectors $1
\otimes 1 \otimes 1$ and $1 \otimes 2 \otimes 1$ in $\B^{\otimes
3}$. Therefore we obtain $\B^{\otimes 3} \simeq \B(3 \epsilon_1)
\oplus \B(2 \epsilon_1 + \epsilon_2)$.

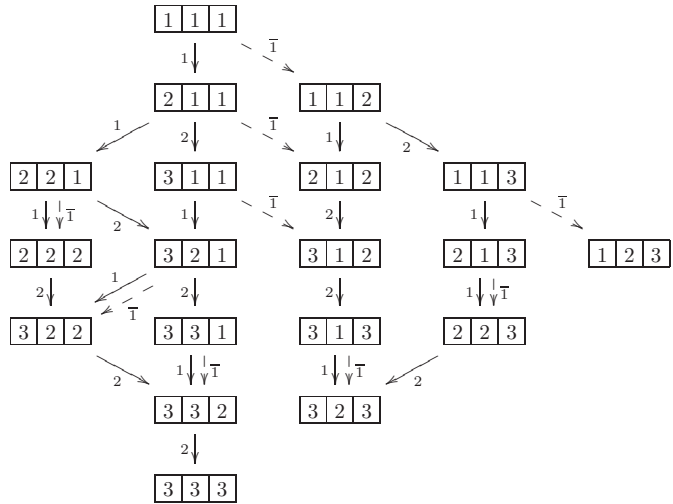
\begin{figure}[h]
 $$\scalebox{.8}{\xymatrix@R=1pc@H=1.5pc{ &  {\young(111)} \ar_1 [d] \ar^{\ol 1}@{-->}[dr] &&&   \\
& {\young(211)} \ar_1[dl] \ar_2[d]  \ar^{\ol 1}@{-->}[dr]  & {\young(112)}   \ar_1[d] \ar_2[dr]  & &    \\
 {\young(221)} \ar@<-0.5ex>_1 [d] \ar@<1ex>^{\ol 1}@{-->}[d] \ar[dr]_2 & {\young(311)} \ar_1[d] \ar^{\ol 1}@{-->}[dr]   & {\young(212)} \ar_2[d]& {\young(113)} \ar[d]_1 \ar^{\ol 1}@{-->}[dr]&  \\
  {\young(222)}\ar[d]_2 & {\young(321)}  \ar@<-0.5ex>_1
[dl] \ar@<1ex>^{\ol 1}@{-->}[dl]  \ar[d]_2  & {\young(312)}  \ar[d]_2 & {\young(213)} \ar@<-0.5ex>_1 [d] \ar@<1ex>^{\ol 1}@{-->}[d] & {\young(123)}   \\
{\young(322)} \ar[dr]_2 & {\young(331)} \ar@<-0.5ex>_1[d] \ar@<1ex>^{\ol 1}@{-->}[d]  & {\young(313)} \ar@<-0.5ex>_1 [d] \ar@<1ex>^{\ol 1}@{-->}[d] & {\young(223)}\ar[dl]^2 &   \\
& {\young(332)}\ar[d]_2      & {\young(323)} & & \\
&{\young(333)} &&& }}$$
 \caption{${\mathbf B}(3 \epsilon_1)$ for $n=3$.} \label{fi:B(3,0,0)}
   \end{figure}

\begin{figure}[!h]
$$\scalebox{.8}{\xymatrix@R=1pc@H=1.5pc{  & {\young(21,:1)} \ar@<-0.5ex>_1[d] \ar@<1ex>^{\ol 1}@{-->}[d] \ar[dr]_2  &  \\
 & {\young(22,:1)} \ar[d]_2 &  {\young(31,:1)} \ar_{\ol 1}@{-->}[dl] \ar[d]_1 \\
  & {\young(32,:1)} \ar[d]_2 &  {\young(31,:2)} \ar@<-0.5ex>_1[d] \ar@<1ex>^{\ol 1}@{-->}[d]  \\
 &{\young(33,:1)} \ar@<-0.5ex>_1[d] \ar@<1ex>^{\ol 1}@{-->}[d] &  {\young(32,:2)} \ar[dl]_2 \\
 & {\young(33,:2)}  & }}$$
 \caption{$\B(2 \epsilon_1 +\epsilon_2)$ for $n=3$.} \label{fi:B(2,1,0)}
   \end{figure}
 \ena
\end{example}

Now the natural question is how to decompose $\mathbf B(\lambda)
\otimes \mathbf B(\mu)$ into a disjoint union of connected
components.
 We define $\la \leftarrow j$ to be the array of boxes obtained
from the shifted shape $\la$ by adding a box at the $j$-th row.
Let us denote by $\la \leftarrow j_1 \leftarrow  \cdots \leftarrow
j_r$ the array of boxes obtained from $\la \leftarrow j_1
\leftarrow \cdots \leftarrow j_{r-1}$ by adding a box at the
$j_r$-th row. We define $\B(\la \leftarrow j_1 \leftarrow \cdots
\leftarrow j_r)$ to be the empty set unless $\la \leftarrow j_1
\leftarrow \cdots \leftarrow j_k$ is a shifted shape for all
$k=1,\ldots,r$.

\begin{theorem} \cite[Theorem 2.8]{GJKKK3} \label{LR rule 1}
Let $\la$ and $\mu$ be strict partitions. Then there is a
$\gq(n)$-crystal isomorphism
\begin{align}
\nn &\B(\la) \otimes \B(\mu) \simeq \soplus_{u_1 u_2 \cdots u_N
\in \B(\la)} \B(\mu \leftarrow(n-u_N+1)  \leftarrow \cdots
\leftarrow (n-u_1+1)),
\end{align}
where $N=|\la|$.
\end{theorem}

\begin{proof}
By the characterization \eqref{ch:lowest wt vector}, the lowest
weight vectors in $\B(\la) \otimes \B(\mu)$ have the form $u_1
\cdots u_N \otimes L^{\mu}$ such that $w_0 \mu + \epsilon_{u_N} +
\epsilon_{u_{N-1}} + \cdots + \epsilon_{u_{k}} \in w_0 (\Lambda^+
\cap P^{\ge 0})$ for all $k=1, \ldots, N$. Hence, the weights of
the highest weight vectors in $\B(\la) \tensor \B(\mu)$ are of the
form $\mu \leftarrow(n-u_N+1) \leftarrow \cdots \leftarrow
(n-u_1+1)$ as desired.
\end{proof}

By Theorem \ref{LR rule 1}, we obtain an explicit description of
{\it shifted Littlewood-Richardson coefficients}.

\begin{corollary} \cite[Corollary 2.9]{GJKKK3}
We define
\begin{equation*}
\begin{array}{rl}
\mathcal{LR}_{\la,\mu}^\nu \seteq \{u=u_1\cdots u_N \in \B(\la) \ ; & {\rm (a)} \ \wt(u)= w_0(\nu -\mu) \ \text{and} \\
  {\rm (b)} \ \mu + \epsilon_{n-u_N+1} +  \cdots &+ \epsilon_{n-u_{k}+1} \in  \Lambda^+ \cap P^{\ge 0} \
  \text{for all} \ 1 \leq k \leq N\},
\end{array}
\end{equation*}
and set $f_{\la,\mu}^{\nu} \seteq |\mathcal{LR}_{\la,\mu}^\nu|$.
Then there is a $\qn$-crystal isomorphism
\begin{equation} \label{eq_decomposition}
\B(\la) \otimes \B(\mu) \simeq \bigoplus_{\nu \in \Lambda^+ \cap
P^{\ge 0}} \B(\nu)^{\oplus f_{\la,\mu}^{\nu}}.
\end{equation}
\end{corollary}

\begin{example}
Let $n=3$, $\la=2 \epsilon_1 + \epsilon_2$ and $\mu=3 \epsilon_1$.
For $u_1 u_2 u_3 \in \B(\la)$, if $u_3 =1$ then the array $\mu
\leftarrow (3-u_3+1)$ is not a shifted shape. When $u_1u_2u_3=132$
or $ 133$, $\mu \leftarrow (3-u_3+1) \leftarrow (3-u_2+1)
\leftarrow( 3-u_1+1)$ is not a shifted shape. For the other $u_1 u_2
u_3 \in \B(\lambda)$, $\mu \leftarrow (3-u_3 +1) \leftarrow
(3-u_2+1) \leftarrow (3-u_1 +1)$ is given as follows:
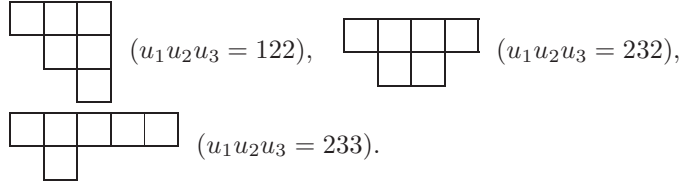
\begin{align*}
&\young(\hfill\hfill\hfill,:\hfill\hfill,::\hfill)  \ \ (u_1u_2u_3=122), \ \ \
\young(\hfill\hfill\hfill\hfill,:\hfill\hfill)  \ \ (u_1u_2u_3=232), \\
& \young(\hfill\hfill\hfill\hfill\hfill,:\hfill) \ \ (u_1u_2u_3=233).
\end{align*}
So we obtain
$$\B(2 \epsilon_1  + \epsilon_2) \otimes \B(3 \epsilon_1) \simeq
\B(3 \epsilon_1  + 2\epsilon_2+\epsilon_3) \op \B(4 \epsilon_1 +
2\epsilon_2) \op \B(5\epsilon_1 + \epsilon_2).$$
\end{example}

As seen in \eqref{eq_decomposition}, the connected component
containing $T \otimes T' \in \B(\la) \otimes \B(\mu)$ is
isomorphic to $\B(\nu)$ for some $\nu$. In order to find $\nu$ and
the element
$S$ of $\B(\nu)$ 
corresponding to $T \otimes T'$ explicitly, we define the {\it
insertion scheme} for semistandard decomposition tableaux.

For an abstract $\gq(n)$-crystal $B$ and an element $b \in B$,
we denote the connected component of $b$ in $B$ by $C(b)$ .
\begin{definition}
Let $B_i$ be an abstract $\qn$-crystals and let $b_i \in B_i$
$(i=1,2)$. We say that $b_1$ is \emph{$\qn$-crystal equivalent to}
$b_2$ if there exists an isomorphism of crystals
$$ C(b_1) \isoto C(b_2)$$
sending $b_1$ to $b_2$. We denote this equivalence relation by
$b_1 \sim b_2$.
\end{definition}

The following $\gq(n)$-crystal equivalence, which is called the
{\em queer Knuth relation}, can be verified in a straightforward
manner.

\begin{proposition}  \cite[Proposition 3.3]{GJKKK3} {\rm (cf. \cite[Theorem 1.4]{Serra})}  \label{prop:Knuth relation}
Let $B_1$ and $B_2$ be the connected components containing $1121$
and $1211$ in $\B^{\otimes 4}$, respectively. Then there exists a
$\qn$-crystal isomorphism $\psi : B_1 \to B_2$ such that
\begin{eqnarray}
\psi(abcd)&
=acbd & \text{if} \ d \leq b \leq a < c \label{eq_Knuth_A} \\
     && \text{or} \ b < d \leq a < c  \label{eq_Knuth_B} \\
     && \text{or} \ b \leq a < d \leq c \label{eq_Knuth_C} \\
     && \text{or} \ a < b < d \leq c, \label{eq_Knuth_D}  \allowdisplaybreaks\\
&=bacd & \text{if} \ b < d \leq c \leq a \label{eq_Knuth_E} \\
     && \text{or} \ d \leq b < c \leq a, \label{eq_Knuth_F}\allowdisplaybreaks\\
&=abdc & \text{if} \ a < d \leq b < c \label{eq_Knuth_G} \\
     && \text{or} \ d \leq a < b < c. \label{eq_Knuth_H}
\end{eqnarray}
\end{proposition}

\begin{definition}(cf. \cite[Definition 2.18]{Serra})
Let $T$ be a semistandard decomposition tableau of shifted shape
$\la$. For $x \in \B$, we define $T \leftarrow x$ to be a filling
of an array of boxes obtained from $T$ by applying the following
procedure: \begin{enumerate}[{\rm (1)}]
\item Let $v_1=u_1 \cdots u_m$ be the reading word of the first row of $T$
such that $u_1 \geq \cdots \geq u_k < \cdots < u_m$ for some $1
\leq k \leq m$. If $v_1 x$ is a hook word, then put $x$ at the end
of the first row and stop the procedure.

\item Assume that $v_1 x$ is not a hook word.
Let $u_j$ be the leftmost element in $v_1 \uparrow$ which is
greater than or equal to $x$. Replace $u_j$ by $x$. Let $u_i$ be
the leftmost element in $v_1 \downarrow$ which is strictly less
than $u_j$. Replace $u_i$ by $u_j$. (Hence $u_i$ is bumped out of
the first row.)

\item Apply the same procedure for the second row with $u_i$ as described in (1) and (2).

\item Repeat the same procedure row by row from top to bottom
until we place a box at the end of a row of $T$. \end{enumerate}
\end{definition}

\begin{example}
Since
\begin{align}
\nonumber \young(66135) \leftarrow 2 = \young(66325,:1), &
  \qquad \young(324) \leftarrow 1 = \young(421,:3),
\end{align}
we obtain
\begin{align}
\young(66135,:324) \leftarrow 2 =\young(66325,:421,::3). \nonumber
\end{align}
\end{example}

Let $T$ and $T'$ be semistandard decomposition tableaux. We define
$T \leftarrow T'$ to be
$$(\cdots((T \leftarrow u_1) \leftarrow u_2) \cdots )\leftarrow u_N,$$
where $u_1 u_2 \cdots u_N$ is the reading word of $T'$.

\begin{example} \label{T <- T'}
\begin{align*}
\young(22,:1) \leftarrow \young(333) &=  \Big(\Big( \Big(\young(22,:1)\leftarrow 3\Big) \leftarrow 3\Big)\leftarrow 3\Big) \\
& =  \Big(\Big(\young(223,:1) \leftarrow 3 \Big) \leftarrow 3\Big) =  \Big(\young(323,:12) \leftarrow 3\Big) \\
&  = \young(333,:22,::1) \ \ ,
\end{align*}
\end{example}

\begin{proposition} \cite[Proposition 3.13, Corollary 3.14]{GJKKK3}  \label{pro_T<-x}
\begin{enumerate}[{\rm (1)}]
\item $T \otimes T'$ is $\qn$-crystal equivalent to $T \leftarrow T'$.

\item $T \leftarrow T'$ is a semistandard decomposition tableau.
\end{enumerate}
\end{proposition}

\begin{proof}
The first assertion follows from the queer Knuth relation. For the second assertion, it
suffices to show that $b_1 \otimes b_2 \leftarrow x$ is a
semistandard decomposition tableau for any $x \in \B$ and $b_1
\otimes b_2 \in \B(\la_1 \epsilon_1 + \la_2 \epsilon_2)$ with $\la_1
> \la_2$. Through a careful investigation on the direct summands in the various tensor products,
one conclude that $b_1 \otimes b_2 \leftarrow x$ lies in
$\B((\la_1 +1) \epsilon_1 + \la_2 \epsilon_2)$ or $\B(\la_1
\epsilon_1 + (\la_2 +1) \epsilon_2)$ or $\B(\la_1 \epsilon_1 +
\la_2 \epsilon_2 + \epsilon_3)$, as desired.
\end{proof}

Using the characterization \eqref{ch:lowest wt vector} of the
lowest weight vectors and Proposition \ref{pro_T<-x}, we obtain
the following theorem.

\begin{theorem} \cite[Theorem 3.15]{GJKKK3}
Let $\la$ and $\mu$ be strict partitions. Then there is a
$\qn$-crystal isomorphism
$$\B(\la) \otimes \B(\mu) \simeq \bigoplus_{\stackrel{T \in \B(\la) \ ; }{T \leftarrow L^{\mu} = L^{\nu}
\ \text{for some} \ \nu \in \Lambda^+ \cap P^{\ge 0}}} \B({\rm
sh}(T \leftarrow L^{\mu})).$$
\end{theorem}

\begin{example} \label{ex_decomposition_2}
Let $n=3$, $\la=2\epsilon_1 +\epsilon_2$ and $\mu=3\epsilon_1$. By Example \ref{T <- T'}, we
get $$\young(22,:1) \leftarrow L^{3 \epsilon_1}=L^{3 \epsilon_1  +
2\epsilon_2+\epsilon_3},$$
and similarly we have
$$\young(32,:2) \leftarrow L^{3\epsilon_1}=L^{4 \epsilon_1 +2 \epsilon_2}, \ \ \ \young(33,:2) \leftarrow L^{3 \epsilon_1} =L^{5 \epsilon_1 +\epsilon_2}.$$

From easy calculations, we know that except the above cases, there is no other
tableau $T \in \B(\la)$ such that $T \leftarrow L^{3
\epsilon_1}=L^{\nu}$ for some strict partition $\nu$. Hence we
conclude that
$$\B(2 \epsilon_1 +\epsilon_2) \otimes \B(3 \epsilon_1) \simeq \B(3 \epsilon_1  + 2\epsilon_2+\epsilon_3) \op \B(4 \epsilon_1 + 2\epsilon_2) \op
\B(5\epsilon_1 + \epsilon_2).$$
\end{example}


\vskip3mm
\bibliographystyle{amsalpha}

\end{document}